\documentclass[12pt]{amsart}
\usepackage[hmargin=2.5cm,vmargin=2.5cm]{geometry}
\usepackage{amsfonts}
\usepackage{amsthm}
\usepackage[english]{algorithm2e}
\usepackage{amsmath}
\usepackage{amssymb}
\usepackage{hyperref}
\usepackage[T1]{fontenc}
\newtheorem{Theorem}{Theorem}[section]
\newtheorem{theorem}[Theorem]{Theorem}
\newtheorem{definition}[Theorem]{Definition}
\newtheorem{remark}[Theorem]{Remark}
\newtheorem{example}[Theorem]{Example}
\newtheorem{lemma}[Theorem]{Lemma}
\newtheorem{proposition}[Theorem]{Proposition}
\newtheorem{corollary}[Theorem]{Corollary}

\title[Cohomologies and linear deformations of relative Rota-Baxter operators+]{Cohomologies and linear deformations of relative Rota-Baxter operators on (pre-)Jacobi-Jordan algebras}
\author[N. Oro Djibril and S. Attan ]
{Nabil Oro Djibril and Sylvain Attan }
\address{Nabil Oro Djibril \newline
	Institut de Math\'ematiques et de Sciences Physiques, Universit\'{e} d'Abomey-Calavi
	01 BP 613-Oganla, Porto-Novo, B\'{e}nin}
\email{nabil.orodjibril@imsp-uac.org}
\address{Sylvain Attan \newline
	D\'{e}partement de Math\'{e}matiques, Universit\'{e} d'Abomey-Calavi
	01 BP 4521, Cotonou 01, B\'{e}nin}
\email{syltane2010@yahoo.fr}
\begin{document}
	\maketitle
	\begin{abstract}
		Some results on (pre-)Jacobi-Jordan algebras and their representations are proved. Moreover, the notion of matched pairs and relative Rota-Baxter operators on these algebras are introduced  and studied. The cohomology theory of relative Rota-Baxter operators on (pre)-Jacobi-Jordan algebras is introduced. We use the cohomological approach to study linear  deformations of relative Rota-Baxter operators. In particular, the notion of Nijenhuis elements is introduced
		to characterize trivial linear deformations. 
	\end{abstract}
	
	\section{Introduction}
	Commutative algebras that satisfy the Jacobi identity rather than the Jordan identity, are called Jacobi-Jordan algebras \cite{dbaf}. They are motivated by the fact that they constitute an interesting sub-category  of the well-known category of Jordan algebras introduced to explain some aspects of physics \cite{pjjvn}. Roughly speaking, Jacobi-Jordan algebras seem appeared first in \cite{kaz} and since then, different names such as mock-Lie algebras \cite{egmk}, Jordan algebras of nil index 3 \cite{sw}, Lie-Jordan algebras \cite{sonk} are  used for these algebraic structures. One can note that
	there is a close relationship between Jacobi-Jordan algebras and (anti)-associative algebras, indeed, they can be generated by 
	(anti)-associative algebras. Recall that an anti-associative algebra is an
	algebra satisfying the identity $(xy)x+ x(yz)=0$ for all $x,y,z.$.  Thanks to the approach of Eilenberg \cite{nj},  representations of Jacobi-Jordan algebras are introduced in \cite{pzu} whereas  a cohomology and deformation theories of Jacobi-Jordan algebras are developped \cite{aBsBaMsM}. The authors construct a cohomology based on two operators. It is called a zigzag cohomology since it deals with two types of cochains and two sequences of operators. Very close to Jacobi-Jordan algebras are pre-Jacobi-Jordan algebras. These are a  recent structure in the area of non-associative algebras, blending ideas from pre-Lie, Jacobi, and Jordan algebras. They are an important subclass of the class of Jacobi-Jordan-admissible algebras studied in \cite{bbmm} where various constructions of these algebras are provided. In recent years, some researchers have become interested in the study of these algebras. In this sense, the cohomology theory of pre-Jacobi-Jordan algebras is introduced \cite{aod} and is called a zigzag cohomology since the cochain complex consists of two sequences of applications. 
	
		The notion of Rota-Baxter algebras was first introduced by Baxter in his study of the fluctuation theory in
	probability \cite{gb}. Since then, much further work on this notion has been done. Indeed, Baxter's work was further investigated by Rota and Cartier respectively in \cite{gcr},\cite{pc} and many excellent results as regards concerning Rota-Baxter algebras are done by Guo
	and al. \cite{lg,lgwk,lgwk1}. Later, Tang, Bai, Guo and Sheng \cite{tbgs} studied deformation theory and cohomology theory of relative Rota-Baxter operators on Lie algebras, with applications to Rota-Baxter Lie algebras in mind. Next,  relative Rota-Baxter operators on Leibniz algebras were studied in \cite{ysr} and their deformation and cohomology theories were obtained in \cite{tsz}. Observe that the deformation of algebraic structures began with the seminal work of Gerstenhaber \cite{ger1,ger2} for associative algebras and followed by its extension to Lie algebras by Nijenhuis and Richardson \cite{nr1,nr2}. In general, deformation theory was developed for binary quadratic operads by Balavoine \cite{db}. Deformations of morphisms and $\mathcal{O}$-operators (also called relative Rota-Baxter operators) were developed in \cite{abm, das,fz1,fz2,tbgs}. 
	Up to now, there was very few study about (relative) Rota-Baxter operators  on (pre)-Jacobi-Jordan algebras.  Thus it is time  to move in this direction by approaching such a subject. More precisely, we are interested in the study of linear deformations of relative Rota-Baxter operators on (pre-)Jacobi-Jordan algebras using the cohomological approach. For this purpose, we first
	define the cohomology of relative Rota-Baxter operators on these algebras. Then we study
	linear deformations and introduce the notion of  Nijenhuis elements associated to a relative
	Rota Rota-Baxter operator, which can give rise to trivial linear deformations. 	
	
	The outline of the paper is as follows: In section 2, we give some basic notions and concepts, namely some properties of (pre-)Jacobi-Jordan algebras and their representations. Here, the notion of matched pair of Jacobi-Jordan algebras is provided and a necessary and sufiscient condition to obtain it, is given. 
	In Section 3 and Section 4, we first introduce the notion of relative Rota-Baxter operators on (pre-)Jacobi-Jordan algebras with respect to representations. Given a relative Rota-Baxter operator, there is a natural (pre-)Jacobi-Jordan algebra on the representation space. We define the cohomology theory of a relative Rota-Baxter operator on a (pre-)Jacobi-Jordan algebra in terms of the cohomology of the (pre-)Jacobi-Jordan algebra on the representation space. Next, we study the linear deformation theory of relative Rota-Baxter operators on (pre-)Jacobi-Jordan algebras. Moreover, we introduce the notion of  Nijenhuis elements associated to a relative Rota-Baxter operator, which gives rise to a trivial linear deformation of the relative Rota-Baxter operator. Finally, we  build a relationship between linear deformations of relative Rota-Baxter operators and linear deformations of the underlying (pre-)Jacobi-Jordan  algebras.
	
	Throughout this paper, all vector spaces  are finite dimensional and meant over a ground field $\mathbb{K}$ of characteristic 0.
	\section{Basic results on representations of (pre-)Jacobi-Jordan algebras}
	This section is devoted to some results on representations  of (pre-)Jacobi-Jordan algebras. The notion on matched pair of (pre-)Jacobi-Jordan algebra is introduced. The fundamental result says that, given  a matched pair of two left pre-Jacobi-Jordan algebras, a matched pair of their sub-adjacent Jacobi-Jordan algebras can be obtained.
	\subsection{Matched pairs of Jacobi-Jordan algebras}
	\begin{definition}
		Let $(A,\mu)$ be a algebra and 
		$\lambda\in\mathbb{K}.$  Let $R$ be a linear map satisfying 
		\begin{eqnarray}
			\mu(R(x),R(y))=R(\mu(R(x),y)+\mu(x,R(y))+\lambda\mu(x,y)),\ \forall x,y \in A. \label{Rota-Baxt}
		\end{eqnarray}
		Then, $R$ is called a Rota-Baxter operator of weight $\lambda$ and $(A,\mu, R)$ is called a Rota-Baxter algebra of weight 
		$\lambda.$
	\end{definition}
	\begin{definition}
		Let $(A,\ast)$ be an algebra.
		\begin{enumerate}
			\item 	A linear map $J: A^{\otimes 3} A$ defined by $$J(x,y,z):=(x\ast y)\ast z+(y\ast z)\ast x+(z\ast x)\ast y$$
			is called a jacobian of $A$ .
			\item The algebra $(A,\ast)$ is called a Jacobi-Jordan algebra if
			\begin{eqnarray}
				&& x\ast y=y\ast x \mbox{ ( commutativity)},\nonumber\\
				&& J(x,y,z)=0\label{JJi}, \mbox{for all $x,y,z\in A.$ }
			\end{eqnarray}
		\end{enumerate}
	\end{definition}
	\begin{example}
		There is a Jacobi-Jordan algebra $(A,\ast)$     
		\cite{dbaf} where non-zero products with respect to a basis $(e_1,e_2,e_3,e_4)$ are given by: $e_1\ast e_1:=e_2;\ e_1\ast e_4=e_4\ast e_1:=e_4.$
	\end{example}
	\begin{proposition}
		Let $(A, \ast)$ be a Jacobi-Jordan algebra and $(L,\bullet)$ be a commutative associative algebra. Then  $(A\otimes L, \circ)$ is a Jacobi-Jordan algebra where
		\begin{eqnarray}
			&& (x\otimes a)\circ(y\otimes b)=(x\ast y)\otimes(a\bullet b)  \mbox{  $\forall x,y\in A$ and $a,b\in L. $}\nonumber
		\end{eqnarray}
	\end{proposition}
	\begin{proof}
		First, the commutativity of $\circ$ follows from the one of $\ast$ and $\bullet$. Next, pick $x,y,z\in A$ and $a,b,c\in L.$ Then,
		\begin{eqnarray}
			&& \Big((x\otimes a)\circ(y\otimes b) \Big)\circ (z\otimes c)+
			\Big((y\otimes b)\circ(z\otimes c) \Big)\circ (x\otimes a)+
			\Big((z\otimes c)\circ(x\otimes a) \Big)\circ (y\otimes b)\nonumber\\
			&&=\Big((x\ast y)\otimes(a\bullet b) \Big)\circ (z\otimes c)+
			\Big((y\ast z)\otimes(b\bullet c) \Big)\circ (x\otimes a)\nonumber\\
			&&+
			\Big((z\ast x)\otimes(c\bullet a) \Big)\circ (y\otimes b)
			=\Big((x\ast y)\ast z\Big)\otimes\Big((a\bullet b)\bullet c\Big)\nonumber\\
			&&+
			\Big((y\ast z)\ast x\Big)\otimes\Big((b\bullet c)\bullet a\Big)
			+\Big((z\ast x)\ast y\Big)\otimes\Big((c\bullet a)\bullet b\Big)\nonumber\\
			&&=\Big((x\ast y)\ast z+(y\ast z)\ast x+(z\ast x)\ast y\Big)\otimes\Big((a\bullet b)\bullet c\Big)\nonumber\\
			&& \mbox{  ( by the commutativity  and the associativity of $(A,\bullet)$ )}\nonumber\\
			&&=0 \mbox{ ( by (\ref{JJi}) ).}\nonumber
		\end{eqnarray}
		Hence, $(A\otimes L, \circ)$ is a Jacobi-Jordan algebra.
	\end{proof}
	Now, we give the definition of representations of a Jacobi-Jordan algebra.
	\begin{definition}\cite{pzu}
		A representation of a Jacobi-Jordan algebra $(A, \ast)$  is a couple $(V,\rho)$ where $V$ is a vector space and $\rho: A\rightarrow gl(V)$ is a linear map
		such that the following identity holds:
		\begin{eqnarray}
			\rho(x\ast y)=-\rho(x)\rho(y)-\rho(y))\rho(x) \mbox{  for all $x,y\in A.$ }\label{rHJJ2}
		\end{eqnarray}
	\end{definition}
	\begin{proposition}\label{PEx0}
		Let  $f:(A, \ast)\rightarrow (B, \bullet)$ be a morphism of Jacobi-Jordan algebras. Then $B^f:=(B, \rho)$ is a representation of $(A,\ast)$ where
		$\rho(a)b:=f(a)\bullet b$ for all $(a,b)\in A\times B.$
	\end{proposition}
	\begin{proof}
		For all $(x,y)\in A^{\times 2}$ and $b\in B,$ since $f$ is a morphism we have:
		\begin{eqnarray}
			&&\rho(x\ast y)b=f(x\ast y)\bullet b=(f(x)\bullet f(y))\bullet b\nonumber\\
			&&=-(f(y)\bullet b)\bullet f(x)-(b\bullet f(x))\bullet f(y) \mbox{ ( by (\ref{JJi}) )}\nonumber\\
			&&-f(x)\bullet(f(y)\bullet b)-f(y)\bullet(f(x)\bullet b) \mbox{ ( commutativity of $\bullet$ )}\nonumber\\
			&&-\rho(x)\rho(y)b-\rho(y)\rho(x)b.\nonumber
		\end{eqnarray}
		Hence, $\rho(x\ast y)=-\rho(x)\rho(y)-\rho(y)\rho(x).$
	\end{proof}
	Now, using Proposition \ref{PEx0}, we obtain the following example as applications.
	\begin{example}
		\begin{enumerate}
			\item Let $(A,\ast)$ be a Jacobi-Jordan algebra. Define a left multiplication 
			$ L: A\rightarrow gl(A)$  by $L(x)y:=x\ast y$  for all $x, y \in A.$ Then $(A , L)$ is a representation of $(A, \ast),$  called a regular representation.
			\item Let $(A,\ast)$ be a Jacobi-Jordan algebra and $B$ be an ideal of $(A,\ast).$ Then $B$  inherits a structure of representation of $(A,\ast)$ where 
			$\rho(a)b:=a\ast b$ for all $(a,b)\in A\times B.$
		\end{enumerate}
	\end{example}
	It is easy to prove.
	\begin{proposition}
		Let $(V,\rho)$	be a representation of a Jacobi-Jordan algebra $(A,\ast)$. Then $(V^{\ast},\pi)$ is a representation of $(A,\ast)$ (called the dual representation) where
		\begin{eqnarray}
			(\pi(x) f)(v):=f(\rho(x)(v)) \mbox{ $\forall (x,f,v)\in A\times V^{\ast}\times V$.}
		\end{eqnarray}
	\end{proposition}
	\begin{proof}
		Let $x,y\in A$, $f\in V^{\ast}$ and $v\in V.$ Then, we have
		\begin{eqnarray}
			&&(\pi(x)\pi(y)f)(v)+(\pi(y)\pi(x) f)(v)\nonumber\\
			&&=
			(f\pi(y)\pi(x))(v)+(f\pi(x)\pi(y))(v)=
			-(\pi(x\ast y)f)(v).\nonumber
		\end{eqnarray}
		Hence, the conclusion follows.
	\end{proof}
	Let recall the following result which is very useful.
	\begin{proposition}\label{spHJJ}\cite{absb}
		Let $(A,\ast)$ be a 
		Jacobi-Jordan algebra. Then $(V, \rho)$ is a representation of $(A, \ast)$ if and only if the direct sum of vector spaces $A\oplus V,$  turns into a 
		Jacobi-Jordan
		algebra with the multiplication defined by
		\begin{eqnarray}
			(x+u)\diamond (y+v):=x\ast y+(\rho(x)v+\rho(y)u).\label{SdpHJJ1}
		\end{eqnarray}
		This Jacobi-Jordan algebra is called the semi-direct product of $A$ with $V$ and is denoted by $A\ltimes V.$
	\end{proposition}
	\begin{definition}
		A matched pair of Jacobi-Jordan algebras denoted by $(A, B,\rho_1,\rho_2),$ consists of two Jacobi-Jordan algebras $(A,\ast)$ and 
		$(B,\bullet)$ together with representations $\rho_1: A\rightarrow gl(B)$ and 
		$\rho_2: B\rightarrow gl(A)$  such that   the following conditions hold for all $x,y\in A,$ $a,b\in B$:
		\begin{eqnarray}
			&& \rho_1(x)(a\bullet b)+(\rho_1(x)a)\bullet b+(\rho_1(x)b)\bullet a\nonumber\\
			&&+\rho_1(\rho_2(a)x)b+\rho_1(\rho_2(b)x)a=0,\label{mpHJJ1}\\
			&& \rho_2(a)(x\ast y)+(\rho_2(a)x)\ast y+(\rho_2(a)y)\ast x\nonumber\\
			&&+\rho_2(\rho_1(x)a)y+\rho_2(\rho_1(y)a)x=0.\label{mpHJJ2}
		\end{eqnarray}
	\end{definition}
	\begin{theorem}
		Let $(A,\ast)$ and 
		$(B,\bullet)$ be Jacobi-Jordan algebras. Then,
		$(A,B,\rho_1,\rho_2),$ is a matched pair of Jacobi-Jordan algebras if and only if  $(A\oplus B, \triangleleft)$ is a Jacobi-Jordan algebra where
		\begin{eqnarray}
			(x+a)\triangleleft(y+b)&:=&(x\ast y+\rho_2(a)y+\rho_2(b)x)+(a\bullet b+\rho_1(x)b+\rho_1(y)a) \label{opHJJ1}
		\end{eqnarray}
	\end{theorem}
	\begin{proof} First, the commutativity of $\triangleleft$ is equivalent to the one of $\ast$ and $\bullet$. Next, for all $x,y,z\in A$ and $a,b,c\in B,$ we compute
		\begin{eqnarray}
			&&\circlearrowleft_{(x+a,y+b,z+c)}\Big((x+b)\triangleleft(y+b) \Big)\triangleleft(z+c)\nonumber\\
			&&=
			\circlearrowleft_{(x+a,y+b,z+c)}\Big( (x\ast y)\ast z+(\rho_2(a)y)\ast z+(\rho_2(b)x)\ast z+\rho_2(a\bullet b)z\nonumber\\
			&&+\rho_2(\rho_1(x)b)z+\rho_2(\rho_1(y)a)z+\rho_2(c)(x\ast y)+\rho_2(c)\rho_2(a)y+\rho_2(c)\rho_2(b)x\Big)\nonumber\\
			&&+\circlearrowleft_{(x+a,y+b,z+c)}\Big( (a\bullet b)\bullet c+(\rho_1(x)b)\bullet c+(\rho_1(y)a)\bullet c+\rho_1(x\ast y)c\nonumber\\
			&&+\rho_1(\rho_2(a)y)c+\rho_1(\rho_2(b)x)c+\rho_1(z)(a\bullet b)+\rho_1(z)\rho_1(x)b+\rho_1(z)\rho_1(y)a\Big)\nonumber
		\end{eqnarray}
		\begin{eqnarray}
			&&=\Big( \circlearrowleft_{(x,y,z)}(x\ast y)\ast z\Big)+\Big( \rho_2(a\bullet b)z+\rho_2(a)\rho_2(b)z+\rho_2(b)\rho_2(a)z\Big)\nonumber\\
			&&+\Big( \rho_2(b\bullet c)x+\rho_2(b)\rho_2(c)x+\rho_2(c)\rho_2(b)x\Big)+
			\Big( \rho_2(c\bullet a)y+\rho_2(c)\rho_2(a)y\nonumber\\
			&&+\rho_2(a)\rho_2(c)y\Big)+\Big( \rho_2(c)(x\ast y)+(\rho_2(c)x)\ast y+(\rho_2(c)y)\ast x +\rho_2(\rho_1(x)c)y\nonumber\\
			&&+\rho_2(\rho_1(y)c)x\Big)
			+\Big( \rho_2(b)(z\ast x)+(\rho_2(b)z)\ast x+(\rho_2(b)x)\ast z +\rho_2(\rho_1(z)b)x\nonumber\\
			&&+\rho_2(\rho_1(x)b)z\Big)
			+\Big( \rho_2(a)(y\ast z)+(\rho_2(a)y)\ast z+(\rho_2(a)z)\ast y +\rho_2(\rho_1(y)a)z\nonumber\\
			&&+\rho_2(\rho_1(z)a)y\Big)+\Big( \circlearrowleft_{(a,b,c)}(a\bullet b)\bullet c\Big)+\Big( \rho_1(x\ast y)c+\rho_1(x)\rho_1(y)c+\rho_1(y)\rho_1(x)c\Big)\nonumber\\
			&&+\Big( \rho_1(y\ast z)a+\rho_1(y)\rho_1(z)a+\rho_1(z)\rho_1(y)a\Big)+
			\Big( \rho_1(z\ast x)b+\rho_1(z)\rho_1(x)b\nonumber\\
			&&+\rho_1(x)\rho_1(z)b\Big)+\Big( \rho_1(z)(a\bullet b)+(\rho_1(z)a)\bullet b+(\rho_1(z)b)\bullet a +\rho_1(\rho_2(a)z)b\nonumber\\
			&&+\rho_1(\rho_2(b)z)a\Big)
			+\Big( \rho_1(y)(c\bullet a)+(\rho_1(y)c)\bullet a+(\rho_1(y)a)\bullet c +\rho_1(\rho_2(c)y)a\nonumber\\
			&&+\rho_1(\rho_2(a)y)c\Big)
			+\Big( \rho_1(x)(b\bullet c)+(\rho_1(x)b)\bullet c+(\rho_1(x)c)\bullet b +\rho_1(\rho_2(b)x)c\nonumber\\
			&&+\rho_1(\rho_2(c)x)b\Big).\nonumber
		\end{eqnarray}
		Therefore, by (\ref{JJi}) and (\ref{rHJJ2}), we have
		\begin{eqnarray}
			&&\circlearrowleft_{(x+a,y+b,z+c)}\Big((x+b)\triangleleft(y+b) \Big)\triangleleft(z+c)\nonumber\\
			&&=\Big( \rho_2(c)(x\ast y)+(\rho_2(c)x)\ast y+(\rho_2(c)y)\ast x +\rho_2(\rho_1(x)c)y+\rho_2(\rho_1(y)c)x\Big)\nonumber\\
			&&+\Big( \rho_2(b)(z\ast x)+(\rho_2(b)z)\ast x+(\rho_2(b)x)\ast z +\rho_2(\rho_1(z)b)x
			+\rho_2(\rho_1(x)b) z\Big)\nonumber\\
			&&+\Big( \rho_2(a)(y\ast z)+(\rho_2(a)y)\ast z+(\rho_2(a)z)\ast y +\rho_2(\rho_1(y)a)z
			+\rho_2(\rho_1(z)a)y\Big)\nonumber\\
			&&+\Big( \rho_1(z)(a\bullet b)+(\rho_1(z)a)\bullet b+(\rho_1(z)b)\bullet a +\rho_1(\rho_2(a)z)b
			+\rho_1(\rho_2(b)z)a\Big)\nonumber\\
			&&+\Big( \rho_1(y)(c\bullet a)+(\rho_1(y)c)\bullet a+(\rho_1(y)a)\bullet c +\rho_1(\rho_2(c)y)a
			+\rho_1(\rho_2(a)y)c\Big)\nonumber\\
			&&+\Big( \rho_1(x)(b\bullet c)+(\rho_1(x)b)\bullet c+
			(\rho_1(x)c)\bullet b +\rho_1(\rho_2(b)x)c+\rho_1(\rho_2(c)x)b\Big).\nonumber
		\end{eqnarray}
		Hence, (\ref{JJi}) is satisfied in $A\oplus B$ if and only if (\ref{mpHJJ1}) and (\ref{mpHJJ2}) hold.
	\end{proof}
\subsection{Matched pairs of pre-Jacobi-Jordan algebras}
\begin{definition}
	Let $(A,\cdot)$ be an algebra.
	\begin{enumerate}
		\item  The anti-associator of $(A,\cdot)$ is the map defined by
		\begin{eqnarray}
			Aasso(x,y,z):=(x\cdot y)\cdot z+x\cdot(y\cdot z) 
			\mbox{ for all $x,y,z\in A.$} \label{antiHas}
		\end{eqnarray}
		\item A left pre-Jacobi-Jordan (left skew-symmetric) algebra  is an algebra $(A,\cdot)$ satisfying
		\begin{eqnarray}
			&& Aasso(x,y,z)=-Aasso(y,x,z) \mbox{ for all $x,y,z\in A.$} \label{HpJJi}
		\end{eqnarray}
		i.e., the anti-associator is left skew-symmetric. Actually, (\ref{HpJJi}) is equivalent to
		\begin{eqnarray}
			(x\ast y)\cdot z=-x\cdot(y\cdot z)-y\cdot(x\cdot z)
			\mbox{ for all $x,y,z\in A.$} \label{HpJJi2}
		\end{eqnarray}
		where $x\ast y=x\cdot y+y\cdot x$ for all 
		$x,y\in A.$
		\item If the anti-associator is right skew-symmetric i.e., 
		\begin{eqnarray}
			Aasso(x,y,z)=-Aasso(x,z,y)\label{HpJJir}
		\end{eqnarray}
		or equivalently
		\begin{eqnarray}
			x\cdot(y\ast z)=-(x\cdot y)\cdot z-(x\cdot z)\cdot y \mbox{ for all $x,y,z\in A,$}\nonumber
		\end{eqnarray}
		then, the algebra is said to be a right pre-Jacobi-Jordan ( right skew symmetric) algebra.	
	\end{enumerate}
\end{definition}
\begin{proposition}\label{lpHJJHJJ}
	Let $(A,\cdot)$ be a left  pre-Jacobi-Jordan  algebra. Then the product given by
	\begin{eqnarray}
		x\ast y=x\cdot y+y\cdot x \label{asHJJ}
	\end{eqnarray}
	defines a Jacobi-Jordan algebra structure on $A,$ which is called the associated (or sub-adjacent) Jacobi-Jordan algebra
	of $(A,\cdot)$ denoted by $A^C$.
\end{proposition}
\begin{proof}
	For all $x,y,z\in A,$ we prove (\ref{JJi}) as follows
	\begin{eqnarray}
		&&J(x,y,z)=\circlearrowleft_{(x,y,z)}(x\ast y)\ast z\nonumber\\
		&&=\circlearrowleft_{(x,y,z)}
		\Big((x\cdot y)\cdot z
		+(y\cdot x)\cdot z+z\cdot(x\cdot y)+z\cdot(y\cdot x) \Big)\nonumber\\
		&&=\circlearrowleft_{(x,y,z)}\Big(Aasso(x,y,z)+Aasso(y,x,z) \Big)=0 \mbox{ ( by (\ref{HpJJi}) ).}\nonumber
	\end{eqnarray}
\end{proof}
	\begin{definition}\cite{aod}\label{rpHJJcl}
		A representation of a left pre-Jacobi-Jordan algebra $(A,\cdot)$ on a vector space $V$ consists of a pair $(\rho, \mu),$ where $\rho,\mu: A\rightarrow gl(V )$  are linear maps satisfying:
		\begin{eqnarray}
			&&\rho(x\ast y)=-\rho(x)\rho(y)-\rho(y)\rho(x),
			\label{7}\\
			&&\mu(y)\mu(x)+\mu(x\cdot y)=-\mu(y)\rho(x)-\rho(x)\mu(y), \label{8}
		\end{eqnarray}
		for all $x,y\in A$ where $x\ast y:=x\cdot y+y\cdot x.$
	\end{definition}
	Observe that, Condition (\ref{7}) means that $(V,\rho)$ is a representation of the subadjacent Jacobi-Jordan of $(A,\cdot).$ One can easily prove that $(V,\rho,\mu)$  is a representation of a left pre-Jacobi-Jordan algebra $(A,\cdot)$ if and only if the direct sum $A\oplus V$ of vector spaces turns into a left pre-Jacobi-Jordan algebra under the product
	\begin{eqnarray}
		(x+u)\circledast (y+v):=x\cdot y+(\rho(x)v+\mu(y)u) \mbox{ for all $x,y\in A$ 
			and $u,v\in V.$}\label{spjj}
	\end{eqnarray}
	This left pre-Jacobi-Jordan algebra is called a semi-direct product of $A$ and $V$  denoted  by $A\ltimes V.$
	
	It is easy to prove the following.
	
	\begin{proposition}\label{sumrl}
		Let $(V, \rho,\mu)$ be a representation of a left pre-Jacobi-Jordan algebra $(A,\cdot)$ and $(A,\ast)$
		be its associated Jacobi-Jordan algebra. Then
		$(V,\rho+\mu)$ is a representation of $(A,\ast),$
	\end{proposition}
	\begin{proof}
		Let $(V, \rho,\mu)$ be a representation of a left pre-Jacobi-Jordan algebra $(A,\cdot)$ and $(A,\ast)$ be its associated Jacobi-Jordan algebra.
		
		We know that $A\ltimes V=(A\oplus V,\circledast)$
		is a left pre-Jacobi-Jordan algebra where $\circledast$ is defined as (\ref{spjj}). Consider its associated Jacobi-Jordan algebra $(A\oplus V,\star),$
		where $\star$ is a Jordan product using $\circledast$. Then,	we have
		\begin{eqnarray}
			&&(x+u)\star (y+v)=(x+u)\circledast (y+v)+
			(y+v)\circledast (x+u)\nonumber\\
			&&=x\cdot y+(\rho(x)v+\mu(y)u)+
			y\cdot x+(\rho(y)u+\mu(x)v)\nonumber\\
			&&=x\ast y+\Big((\rho+\mu)(x)v+(\rho+\mu)(y)u \Big).\nonumber
		\end{eqnarray}
		Thanks to Proposition \ref{spHJJ}, we deduce that $(V,\rho+\mu)$ is a representation of $(A,\ast).$
	\end{proof}
	\begin{definition}
		A matched pair of left pre-Jacobi-Jordan algebras denoted by $(A_1,A_2,\rho_1,
		\mu_1,\rho_2,\mu_2),$ consists of two left pre-Jacobi-Jordan algebras $\mathcal{A}_1:=(A_1,\cdot)$ and 
		$\mathcal{A}_2:=(A_2,\top)$ together with representations $\rho_1, \mu_1: A_1\rightarrow gl(A_2)$ and 
		$\rho_2, \mu_2: A_2\rightarrow gl(A_1)$  such that the following conditions hold for all $x,y,z\in A_1,$ $a,b,c\in A_2:$  
		\begin{eqnarray}
			\rho_1(x)(a\intercal b)&=&-\rho_1(\rho_2(a)x+\mu_2(a)x)b-(\rho_1(x)a+
			\mu_1(x)a)\intercal b\nonumber\\
			&&-\mu_1(\mu_2(b)x)a-a\intercal(\rho_1(x)b),\label{mppHJJ1}\\
			\mu_1(x)(a\circledast b)&=&-a\intercal(\mu_1(x)b)-b\intercal(\mu_1(x)a)
			-\mu_1(\rho_2(a)x)b\nonumber\\
			&&-\mu_1(\rho_2(b)x)a,\label{mppHJJ2}\\
			\rho_2(a)(x\cdot y)&=&-\rho_2(\rho_1(x)a+\mu_1(x)a)y-(\rho_2(a)x+
			\mu_2(a)x)\cdot y\nonumber\\
			&&-\mu_2(\mu_1(y)a)x-x\cdot(\rho_2(a)y),\label{mppHJJ3}\\
			\mu_2(a)(x\ast y)&=&-x\cdot(\mu_2(a)
			y)-y\cdot(\mu_2(a)x)
			-\mu_2(\rho_1(x)a)y\nonumber\\
			&&-\mu_2(\rho_1(y)a)x,\label{mppHJJ4}
		\end{eqnarray}
		where $\ast$ is a product of the sub-adjacent Jacobi-Jordan algebra $A_1^C$ and $\circledast$ is a product of the sub-adjacent Jacobi-Jordan algebra $A_2^C.$
	\end{definition}
	\begin{theorem}
		Let $\mathcal{A}_1:=(A_1,\cdot)$ and 
		$\mathcal{A}_2:=(A_2,\top)$ be left pre-Jacobi-Jordan algebras. Then,
		$(A_1,A_2,\rho_1,\mu_1,\rho_2,\mu_2),$ is a matched pair of left pre-Jacobi-Jordan algebras if and only if  $(A_1\oplus A_2, \diamond)$ is a left pre-Jacobi-Jordan algebra where
		\begin{eqnarray}
			(x+a)\diamond(y+b)&:=&(x\cdot y+\rho_2(a)y+\mu_2(b)x)+(a\top b+\rho_1(x)b+\mu_1(y)a). \label{oppHJJ1}
		\end{eqnarray}
	\end{theorem}
	\begin{proof} Let $x,y,z\in A_1$ and $a,b,\in A_2$, then by straightforward computations
		\begin{eqnarray}
			&& Aasso_{A_1\oplus A_2}(x+a,y+b,z+c)\nonumber\\
			&&=((x+a)\diamond(y+b))\diamond(z+c)+
			(x+a)\diamond((y+b)\diamond(z+c))\nonumber\\
			&&=(x\cdot y)\cdot z+(\rho_2(a)y)\cdot z+(\mu_2(b)x)\cdot z+\rho_2(a\intercal b)z+\rho_2(\rho_1(x)b)z\nonumber\\
			&&+\rho_2(\mu_1(y)a)z+\mu_2(c)(x\cdot y)+\mu_2(c)\rho_2(a)y+\mu_2(c)\mu_2(b)x+(a\intercal b)\intercal c\nonumber\\
			&&+(\rho_1(x)b)\intercal c+(\mu_1(y)a)\intercal c+\rho_1(x\cdot y)c+\rho_1(\rho_2(a)y)c
			+\rho_1(\mu_2(b)x)c\nonumber\\
			&&+\mu_1(z)(a\intercal b)+\mu_1(z)\rho_1(x)b+\mu_1(z)\mu_1(y)a
			+x\cdot(y\cdot z)+x\cdot(\rho_2(b)z)\nonumber\\
			&&+x\cdot(\mu_2(c)y)+\rho_2(a)(y\cdot z)+\rho_2(a)\rho_2(b)z+\rho_2(a)\mu_2(c)y+\mu_2(b\intercal c)x\nonumber\\
			&&+\mu_2(\rho_1(y)c)x+\mu_2(\mu_1(z)b)x
			+a\intercal(b\intercal c)+a\intercal(\rho_1(y)c)
			+a\intercal(\mu_1(z)b)\nonumber\\
			&&+\rho_1(x)(b\intercal c)+\rho_1(x)\rho_1(y)c+\rho_1(x)\mu_1(z)b+\mu_1(y\cdot z)a+\mu_1(\rho_2(b)z)a\nonumber\\
			&&+\mu_1(\mu_2(c)y)a.\nonumber
		\end{eqnarray}
		Switching $x+a$ and $y+b$ in the above expression of $Aasso_{A_1\oplus A_2}(x+a,y+b,z+c)$, we obtain  $Aasso_{A_1\oplus A2}(y+b,x+a,z+c).$
		Next, after rearranging terms, we obtain
		\begin{eqnarray}
			&&Aasso_{A_1\oplus A_2}(x+a,y+b,z+c)+Aasso_{A_1\oplus A_2}(y+b,x+a,z+c)\nonumber\\
			&&
			=\Big(Aasso_{A_1}(x,y,z)+Aasso_{A_1}(y,x,z)\Big)
			+\Big(Aasso_{A_2}(a,b,c)+Asso_{A_2}(b,a,c)\Big)\nonumber\\
			&&+
			\Big(\mu_1(z)\mu_1(y)a+\mu_1(y\cdot z)a+\mu_1(z)\rho_1(y)a
			+\rho_1(y)\mu_1(z)a \Big)\nonumber\\
			&&+	\Big(\mu_1(z)\mu_1(x)b+\mu_1(x\cdot z)b+\mu_1(z)\rho_1(x)b+\rho_1(x)\mu_1(z)b \Big)\nonumber\\
			&&+\Big(\mu_2(c)\mu_2(b)x+\mu_2(b\intercal c)x+\mu_2(c)\rho_2(b)x+\rho_2(b)\mu_2(c)x \Big)+
			+\Big(\mu_2(c)\mu_2(a)y\nonumber\\
			&&+\mu_2(a\intercal c)y+\mu_2(c)\rho_2(a)y+\rho_2(a)\mu_2(c)y \Big)+
			\Big(\rho_1(x\ast y)c+\rho_1(x)\rho_1(y)c\nonumber\\
			&&+ \rho_1(y)\rho_1(x)c\Big)+\Big(\rho_2(a\circledast b)z+\rho_2(a)\rho_2(b)z+ \rho_2(b)\rho_2(a)z\Big)\nonumber
		\end{eqnarray}
		\begin{eqnarray}		
			&&+
			\Big( \rho_1(y)(a\intercal c)+\rho_1(\rho_2(a)y+\mu_2(a)y)c+(\rho_1(y)a+\mu_1(y)a)\intercal c+\mu_1(\mu_2(c)y)a\nonumber\\
			&&+a\intercal(\rho_1(y)c)\Big)
			+\Big( \rho_1(x)(b\intercal c)+\rho_1(\rho_2(b)x+\mu_2(b)x)c+(\rho_1(x)b+\mu_1(x)b)\intercal c\nonumber\\
			&&+\mu_1(\mu_2(c)x)b
			+b\intercal(\rho_1(x)c)\Big)
			+\Big(\mu_1(z(a\circledast b)+a\intercal(\mu_1(z)b)+b\intercal(\mu_1(z)a)\nonumber\\
			&&+\mu_1(\rho_2(a)z)b+\mu_1(\rho_2(b)z)a   \Big)+
			\Big(\rho_2(a)(y\cdot z)+\rho_2(\rho_1(y)a+\mu_1(y)a)z\nonumber\\
			&&+(\rho_2(a)y+\mu_2(a)y)\cdot z+\mu_2(\mu_1(z)a)y+y\cdot(\rho_2(a)z) \Big)+\Big(\rho_2(b)(x\cdot z)\nonumber\\
			&&+\rho_2(\rho_1(x)b+\mu_1(x)b)z+(\rho_2(b)x+\mu_2(b)x)\cdot z+\mu_2(\mu_1(z)b)x+x\cdot(\rho_2(b)z) \Big)\nonumber\\
			&&+\Big(\mu_2(c)(x\ast y)+x\cdot(\mu_2(c)y)+y\cdot(\mu_2(c)x)
			+\mu_2(\rho_1(x)c)y+\mu_2(\rho_1(y)c)x \Big).\nonumber
		\end{eqnarray}
		Thanks to  (\ref{HpJJi}), (\ref{7}) and (\ref{8}), we obtain
		\begin{eqnarray}
			&&Aasso(x+a,y+b,z+c)+Aasso(y+b,x+a,z+c)
			\nonumber\\ 
			&&=\Big( \rho_1(y)(a\intercal c)+\rho_1(\rho_2(a)y+\mu_2(a)y)c+(\rho_1(y)a+\mu_1(y)a)\intercal c+\mu_1(\mu_2(c)y)a\nonumber\\
			&&+a\intercal(\rho_1(y)c)\Big)
			+\Big( \rho_1(x)(b\intercal c)+\rho_1(\rho_2(b)x+\mu_2(b)x)c+(\rho_1(x)b+\mu_1(x)b)\intercal c\nonumber\\
			&&+\mu_1(\mu_2(c)x)b
			+\alpha_2(b)\intercal(\rho_1(x)c)
			\Big)
			+\Big(\mu_1(z(a\circledast b)+a\intercal(\mu_1(z)b)+b\intercal(\mu_1(z)a)\nonumber\\
			&&+\mu_1(\rho_2(a)z)b+\mu_1(\rho_2(b)z)a   \Big)+
			\Big(\rho_2(a)(y\cdot z)+\rho_2(\rho_1(y)a+\mu_1(y)a)z\nonumber\\
			&&+(\rho_2(a)y+\mu_2(a)y)\cdot z+\mu_2(\mu_1(z)a)y+y\cdot(\rho_2(a)z) \Big)+\Big(\rho_2(b)(x\cdot z)\nonumber\\
			&&+\rho_2(\rho_1(x)b+\mu_1(x)b)z+(\rho_2(b)x+\mu_2(b)x)\cdot z+\mu_2(\mu_1(z)b)x+x\cdot(\rho_2(b)z) \Big)\nonumber\\
			&&+\Big(\mu_2(c)(x\ast y)+x\cdot(\mu_2(c)y)+y\cdot(\mu_2(c)x)
			+\mu_2(\rho_1(x)c)y+\mu_2(\rho_1(y)c)x \Big).\nonumber
		\end{eqnarray}
		We deduce that (\ref{HpJJi}) holds in $A_1\oplus A_2$  if and only (\ref{mppHJJ1}), (\ref{mppHJJ2}), (\ref{mppHJJ3}) and (\ref{mppHJJ4}) holds.
	\end{proof}
	\begin{proposition}
		Let $(A_1,A_2,\rho_1,\mu_1,\rho_2,
		\mu_2)$ be a matched pair of left pre-Jacobi-Jordan algebras
		$\mathcal{A}_1:=(A_1,\cdot)$ and $\mathcal{A}_2:=(A_2,\top).$  Then,
		$(A_1, A_2, \rho_1+\mu_1,\rho_2+\mu_2)$ is a matched pair of sub-adjacent Jacobi-Jordan algebras $A_1^C:=(A_1,\ast)$ and $A_2^C:=(A_2,\circledast)$.
	\end{proposition}
	\begin{proof}
		Since $(A_1,\rho_2,\mu_2)$ and $(A_2,\rho_1,\mu_1)$ are representations of left pre-Jacobi-Jordan algebras $(A_2,\intercal)$
		and $(A_1,\cdot)$ respectively, it follows by Proposition \ref{sumrl}  that $(A_1,\theta_2:=\rho_2+\mu_2)$ and $(A_2,\theta_1:=\rho_1+\mu_1)$ are representations of sub-adjacent Jacobi-Jordan algebras $A_2^C$ and $A_1^C$ respectively.
		Next, let $x,y\in A_1$ and $a,b\in A_2.$ Then, by straightforward computations, after rearranging terms we get:
		\begin{eqnarray}
			&& \theta_1(x)(a\circledast b)+(\theta_1(x)a)\circledast b+(\theta_1(x)b)\circledast a
			+\theta_1(\theta_2(a)x)b+\theta_1(\theta_2(b)x)a\nonumber\\
			&&=
			\Big(\rho_1(x)(a\intercal b)+\rho_1(\rho_2(a)x+\mu_2(a)x)b
			+(\rho_1(x)a+\mu_1(x)a)
			\intercal b
			+\mu_1(\mu_2(b)x)a\nonumber\\
			&&+a\intercal(\rho_1(x)b)\Big)
			+ \Big(\rho_1(x)(b\intercal a)+\rho_1(\rho_2(b)x+\mu_2(b)x)a+(\rho_1(x)b+\mu_1(x)b)\intercal a\nonumber\\
			&&
			+\mu_1(\mu_2(a)x)b+b\intercal(\rho_1(x)a)\Big)
			+\Big(\mu_1(x)(a\circledast b)+a\intercal(\mu_1(x)b)+b\intercal(\mu_1(x)a)
			\nonumber\\
			&&+\mu_1(\rho_2(a)x)b+\mu_1(\rho_2(b)x)a\Big)=0 
			\mbox{ ( by (\ref{mppHJJ1}), (\ref{mppHJJ2}) )}.\nonumber
		\end{eqnarray}
		Hence, we obtain (\ref{mpHJJ1}). Similarly, we compute
		\begin{eqnarray}
			&& \theta_2(a)(x\ast y)+(\theta_2(a)x)\ast y+(\theta_2(a)y)\star x
			+\theta_2(\theta_1(x)a)y+\theta_2(\theta_1(y)a)x\nonumber\\
			&&=
			\Big(\rho_2(a)(x\cdot y)+\rho_2(\rho_1(x)a+\mu_1(x)a)y
			+(\rho_2(a)x+\mu_2(a)x)\cdot y
			+\mu_2(\mu_1(y)a)x\nonumber\\
			&&+x\cdot(\rho_2(a)y)\Big)
			+ \Big(\rho_2(a)(y\cdot x)+\rho_2(\rho_1(y)a+\mu_1(y)a)x+(\rho_2(a)y+\mu_2(a)y)\cdot x\nonumber\\
			&&
			+\mu_2(\mu_1(x)a)y+y\cdot(\rho_2(a)x)\Big)
			+\Big(\mu_2(a)(x\ast y)+x\cdot(\mu_2(a)y)+y\cdot(\mu_2(a)x)
			\nonumber\\
			&&+\mu_2(\rho_1(x)a)y+\mu_2(\rho_1(y)a)x\Big)=0 
			\mbox{ ( by (\ref{mppHJJ3}), (\ref{mppHJJ4}) )}.\nonumber
		\end{eqnarray}
		Then, we get also (\ref{mpHJJ2}).
	\end{proof}	
	\section{Cohomologies and linear deformations of a relative Rota-Baxter operators on Jacobi-Jordan algebras.}
	This section  contains most results of this paper. We first introduce the notion of relative Rota-Baxter operators of Jacobi-Jordan algebras  and prove several results about this concepts. Linear deformations of Jacobi-Jordan algebras  and the one of their relative Rota-Baxter operators are also studied and links between these two deformations are given.
	\subsection{Cohomologies of relative Rota-Baxter operators on Jacobi-Jordan algebras}
	\begin{definition}
		Let $(V,\rho)$ be a representation of a Jacobi-Jordan algebra $(A,\ast).$ A linear
		operator $T : V\rightarrow A$ is called a relative Rota-Baxter operator of $A$ with respect to $\rho$ if it satisfies 
		\begin{eqnarray}
			&& T(u)\ast T(v)=T\Big(\rho(T(u))v+\rho(T(v))u\Big) \mbox{ for all $u,v\in V$.} \label{rbHJJ2}
		\end{eqnarray}
	\end{definition}
	Observe  that Rota-Baxter operators on Jacobi-Jordan algebras are relative Rota-Baxter operators with respect to the regular representation.
	\begin{example}
		Consider the $3$-dimensional Jacobi-Jordan algebra $(A,\ast)$ \cite{dbaf} where  nonzero products with respect to a basis $(e_1,e_2,e_3)$ are given by:
		$$e_1\ast e_1=e_2=e_3\ast e_3.$$
		Then the linear map $T: A\rightarrow A$  defined by 
		$$T(e_1):=a_1e_1+a_2e_2+a_3e_3,\ T(e_2):=b_1e_1+b_2e_2+b_3e_3,\ T(e_3):=c_1e_1+c_2e_2+c_3e_3$$ is a relative Rota-Baxter operator with respect to the regular representation if and only if\\
		$$\begin{cases}
			b_1=b_3=0\\
			a_1^2+a_3^2=2a_1b_2\\
			a_1c_1+a_3c_3=a_3b_2+b_2c_1\\
			c_1^2+c_3^2=2b_2c_3
		\end{cases} 
		.$$
		In particular, we get $T=\left(
		\begin{array}{cccc}
			0&0&0\\
			a_2&0&c_2\\
			0&0&0
		\end{array}
		\right)$ 
		and $T=\left(
		\begin{array}{cccc}
			0&0&0\\
			a_2&b_2&c_2\\
			0&0&b_2+|b_2|
		\end{array}
		\right).$
	\end{example}
	\begin{proposition}
		Let $(A, \ast)$ be a Jacobi-Jordan algebra and $(V,\rho)$ be a representation of $(A, \ast).$
		Then, $A\oplus V$ is a representation of $(A, \ast)$ under the maps $\rho_{A\oplus V}: A\rightarrow gl(A\oplus V)$ defined by
		\begin{eqnarray}
			&&\rho_{A\oplus V}(a)(b+v):=(a\ast b)+\rho(a)v.\nonumber
		\end{eqnarray}
		Define the linear map $T: A\oplus V\rightarrow A, a+u\mapsto a.$ Then $T$ is up to a scalar coefficient, a relative Rota-Baxter operator on $A$ with
		respect to the representation $(A\oplus V,\rho_{A\oplus V}).$
	\end{proposition}
	\begin{proof}
		First, let $a,b,c\in A$ and $u\in V.$ Then by the commutativity of $\ast$, identities 
		(\ref{JJi}) and (\ref{rHJJ2}), we obtain
		\begin{eqnarray}
			&&\rho_{A\oplus V}(a\ast b)(c+u)=(a\ast b)\ast c+\rho(a\ast b)u=-(a\ast(b\ast c)+b\ast (a\ast c))-(\rho(a)\rho(b)c+\rho(b)\rho(a)c)\nonumber\\
			&&-\rho_{A\oplus V}(a)\rho_{A\oplus V}(b)(c+u)
			-\rho_{A\oplus V}(b)\rho_{A\oplus V}(a)(c+u).\nonumber
		\end{eqnarray}
		Next, pick $a,b\in A$ and $u,v\in V,$ then $T(a+u)\ast T(b+v)=a\ast b$ and
		\begin{eqnarray}
			T(\rho_{A\oplus V}(T(a+u))(b+v)+\rho_{A\oplus V}(T(b+v))(a+u))=T(a\ast b+\rho(a)v+b\ast a+\rho(b)u)=2(a \ast b).\nonumber
		\end{eqnarray}
		Hence, the conclusion follows.
	\end{proof}
	Let give some characterizations of relative Rota-Baxter operators on Jacobi-Jordan algebras.
	\begin{proposition}
		Let $(V,\rho)$ be a representation  of a Jacobi-Jordan algebra $(A, \ast)$. Then, a linear map $T : V\rightarrow A$ is a relative Rota-Baxter operator with respect to $(V,\rho)$ if and only if the graph of $T,$
		$$G_r(T):=\{(T(v), v), v\in  V\}$$
		is a subalgebra of the semi-direct product algebra $A\ltimes V.$
	\end{proposition}
	\begin{proof}
		Pick $u,v\in V.$ Then, we compute:
		\begin{eqnarray}
			(Tu,u)\diamond (Tv,v)=(Tu\ast Tv, \rho(Tu)v+\rho(Tv)u).\nonumber
		\end{eqnarray}
		Hence 
		\begin{eqnarray}
			(Tu,u)\diamond (Tv,v)\in Gr(T)\Leftrightarrow Tu\ast Tv=T(\rho(Tu)v+\rho(Tv)u).
		\end{eqnarray}
	\end{proof}
	The following result shows that a relative Rota-Baxter operator can be lifted up the
	Rota-Baxter operator.
	\begin{proposition}
		Let $(A, \ast)$ be a Jacobi-Jordan algebra, $(V,\rho)$ be a representation of $A$ and
		$T : V\rightarrow A$ be a linear map. Define 
		$\widehat{T}\in End(A\oplus V)$ by 
		$\widehat{T}(a+v):=Tv.$ Then T is an relative Rota-Baxter operator with respect to to $(V,\rho)$
		if and only if $\widehat{T}$ is a Rota-Baxter operator on $A\oplus V$ of weight $\lambda=0.$
	\end{proposition}
	\begin{proof}
		Let $a,b\in A$ and $u,v\in V.$ Then, we have
		\begin{eqnarray}
			\widehat{T}(a+u)\diamond\widehat{T}(b+v)=Tu\diamond Tv=Tu\ast Tv.\nonumber 
		\end{eqnarray}
		and
		\begin{eqnarray}
			&&\widehat{T}(\widehat{T}(a+u)\diamond(b+v)+(a+u)\diamond\widehat{T}(b+v))=\widehat{T}(Tu\ast b+\rho(Tu)v+a\ast Tv+\rho(Tv)u)\nonumber\\
			&&=T(\rho(Tu)v+\rho(Tv)u)\nonumber.
		\end{eqnarray}
		Hence,
		\begin{eqnarray}
			&&\widehat{T}(a+u)\diamond\widehat{T}(b+v)=\widehat{T}(\widehat{T}(a+u)\diamond(b+v)+(a+u)\diamond\widehat{T}(b+v))\nonumber
		\end{eqnarray}
		if and only if,
		\begin{eqnarray}
			&&Tu\ast Tv=T(\rho(Tu)v+\rho(Tv)u).
		\end{eqnarray}	
	\end{proof}
	\begin{proposition}\label{HJJpJJ}
		Let $(A,\ast)$ be a Jacobi-Jordan algebra and $(V, \rho)$ be a representation. If $T$ is
		a relative Rota-Baxter operator with respect to $\rho$, then $V_T:=(V, \ast_T)$ is a Jacobi-Jordan algebra, where
		\begin{eqnarray}
			u\ast_T v:=\rho(T(u))v+ \rho(T(v))u\mbox{ for $u,v\in V$} \label{revasHJJ1}
		\end{eqnarray}
		and $T$ is a morphism of Jacobi-Jordan algebras.
	\end{proposition}
	\begin{proof}
		Let $u, v, w\in V$ and put $u\cdot v=\rho(T(u))v.$ Then, we have
		$u\ast_T v=u\cdot v + v\cdot u.$  It is then enough to prove thanks to Proposition \ref{lpHJJHJJ} that $\cdot$ satisfies (\ref{HpJJi2}). Note
		first that $T(u\ast_T v) = T(u)\ast T(v).$ Then, we compute 
		(\ref{HpJJi2}) as follows:
		\begin{eqnarray}
			&&(u\ast_T v)\cdot w=\rho(T(u)\ast T(v))w-\rho(T u)\rho(Tv)w-\rho(T v)\rho(tu)w\nonumber\\
			&&=-u\cdot(v\cdot w)-v\cdot(u\cdot w).\nonumber
		\end{eqnarray}
		Therefore, $(V, \ast_T)$ is a  Jacobi-Jordan algebra and $T$ is a morphism of Jacobi-Jordan algebras.
	\end{proof}
	\begin{proposition}
		Let $T : V\rightarrow A$ be a relative Rota-Baxter operator on the Jacobi-Jordan algebra $( A, \ast)$ with respect to the representation $(V,\rho).$
		Let us define a map $\rho_T : V\rightarrow  gl(A)$ given by
		\begin{eqnarray}
			\rho_T(u)x:=T(u)\ast x-T(\rho(x)u) \mbox{ for all $(u,x)\in V\times A $}\nonumber
		\end{eqnarray}
		Then, the couple $(A,\rho_T)$ is a representation of the  Jacobi-Jordan algebra $V_T=(V,\ast_T)$.
	\end{proposition}
	\begin{proof}
		Let $u,v\in V,\ x\in A.$ Recall that $u\ast_T v=\rho(Tu)v+\rho(Tv)u$ and $T(u\ast_T v)=T(u)\ast T(v).$ Then, by  straightforward  computations, we have
		\begin{eqnarray}
			&&\rho_T(u\ast_T v)x=(T(u)\ast T(v))\ast x-T\Big(\rho(x)\rho(Tu)v\Big)-T\Big(\rho(x)\rho(Tv)u\Big).\nonumber
		\end{eqnarray}
		Also, we compute
		\begin{eqnarray}
			&&\rho_T(u)\rho_T(v)x=Tu\ast(Tv\ast x)-Tu\ast T(\rho(x)v)-
			T\Big(\rho(Tv\ast x )u\Big)\nonumber\\
			&&+T\Big( \rho(T(\rho(x)v))u\Big)=
			Tu\ast(Tv\ast x)-T\Big(\rho(Tu)\rho(x)v+\rho(T(\rho(x)v))u \Big)\nonumber\\
			&&-
			T\Big(\rho(T v)\rho(x)u+\rho(x)\rho(Tv)u \Big)+T\Big( \rho(T(\rho(x)v))u\Big)
			\mbox{ (by (\ref{rbHJJ2}) and (\ref{rHJJ2})  ) }\nonumber\\
			&&=
			T u\ast(Tv\ast x)-T\Big(\rho(T u)\rho(x)v\Big)+
			T\Big(\rho(T v)\rho(x)u\Big)+T\Big(\rho(x)\rho(Tv)u \Big).\nonumber
		\end{eqnarray}
		Switching $u$ and $v$ in the above equation, we come to
		\begin{eqnarray}
			\rho_T(v)\rho_T(u)x=
			Tv\ast(Tu\ast x)-T\Big(\rho(T v)\rho(x)u\Big)+
			T\Big(\rho(T u)\rho(x)v\Big)+T\Big(\rho(x)\rho(Tu)v \Big).\nonumber
		\end{eqnarray}
		It follows by (\ref{JJi}) that
		\begin{eqnarray}
			-\rho_T(u)\rho_T(v)x-\rho_T(v)\rho_T(u)x=\rho_T(u\ast_T v)x,\nonumber
		\end{eqnarray}
		i.e., (\ref{rHJJ2}) holds in $(A,\rho_T).$
	\end{proof}

Let consider the Jacobi-Jordan algebra $V_T=(V,\ast_T)$ and its representation $(A,\rho_T)$.
Thanks to the cochain complex  of the so-called zigzag cohomologies of Jacobi-Jordan algebra introduced in \cite{aBsBaMsM},  the corresponding   cochain complex of the Jacobi-Jordan algebra $(V,\ast_T)$ with coefficients in the representation $(A,\rho_T)$ is  given by    $ (C^{\bullet}(V,A), \mathbf{A}^{\bullet}(V,A),d_{\rho_T}^{\bullet},\delta_{\rho_T}^{\bullet}) $   where 
$$\mbox{ for all } n\in \mathbb{N},n\geq 1 ,\   \mathbf{C}^{n}(V,A):=\{f:V^{\otimes n}\rightarrow A, f \text{ is linear}\},\ 
 \mathbf{A}^{n}(V,A)=Hom(\wedge^{n} V, A),$$
 $$d_{\rho_T}^n: C^n(V,A)\rightarrow C^{n+1}(V,A) \mbox{  and } \delta_{\rho_T}^n: A^n(V,A)\rightarrow C^{n+1}(V,A).$$ 
	More precisely,
	$$d_{\rho_T}^0(x)(v)=\delta_{\rho_T}^0(x)(v):=\rho_T(v)(x)=T(v)\ast x-T(\rho(x)v).$$
	Furtheremore, for any $n\in\mathbb{N}\setminus\{0\},$ 
	\begin{eqnarray}
		&& d_{\rho_T}^n f(u_1, \cdots, u_{n+1})=\sum\limits_{i=1}^{n+1}\rho_T(u_i)f(u_1,\cdots,\widehat{u_i},\cdots, u_{n+1})\nonumber\\
		&&+\sum\limits_{1\leq i< j\leq n+1} f(u_i\ast_T u_j, u_1, \cdots, \widehat{u_i}, \cdots, \widehat{u_j}, \cdots, u_{n+1}) 
	\end{eqnarray}
	\begin{eqnarray}
		&& \delta_{\rho_T}^n f(u_1, \cdots, u_{n+1})=\sum\limits_{i=1}^{n+1}\rho_T(u_i)f(u_1,\cdots,\widehat{u_i},\cdots, u_{n+1})\nonumber\\
		&&-\sum\limits_{1\leq i< j\leq n+1} f(u_i\ast_T u_j, u_1, \cdots, \widehat{u_i}, \cdots, \widehat{u_j}, \cdots, u_{n+1}). 
	\end{eqnarray}
and 
\begin{eqnarray}
\forall n\in\mathbb{N}\setminus\{0\}, d_{\rho_T}^n\circ\delta_{\rho_T}^{n-1}=0.
\end{eqnarray}
	\begin{definition} 
		Let $T$ be a relative Rota-Baxter operator of the Jacobi-Jordan algebra $(A,\ast)$ with respect to a representation 
		$(V,\rho).$ The cochain complex  $(C^{\bullet}(V,A)=\bigoplus_{k=0}^{+\infty} C^k(V,A),\mathbf{A}^{\bullet}(V,A)=\bigoplus_{k=0}^{+\infty} \mathbf{A}^k(V,A) ,d_{\rho_T}^{\bullet},\delta_{\rho_T}^{\bullet})$ is
		called the cohomology complex for the relative Rota-Baxter operator $T.$
	\end{definition}
	The $k^{th}$ cohomology space of $V$ with values/coefficients in $A$ is given 
	by the quotient
	\begin{eqnarray}
		H^k(V,A):=Z^k(V,A)/B^k(V,A).
	\end{eqnarray}
	It is obvious that $x\in A$ is closed if and only if $L_x\circ T-T\circ p(x)=0$ and $f\in C^1(V,A)$ is closed if and only if
	$T(u)\ast f(v)+T(v)\ast f(u)-T(\rho(f(u))(v)+\rho(f(v))(u))-f(\rho(T(u))(v)+\rho(T(v))(u))=0.$ 
	\subsection{Linear deformations of a relative Rota-Baxter operators on Jacobi-Jordan algebras.}
	\begin{definition}
		Let $(A,\ast)$ be a Jacobi-Jordan algebra and $ T:V\rightarrow A $ be a relative Rota-Baxter operator with respect to a representation $(V,\rho)$. We say that a linear map  $\mathcal{Z}:V\rightarrow A$ generates a linear deformation of $T$  if \,\  
		 $\forall  t \in \mathbb{K},T_{t}=T+t\mathcal{Z}$ is a relative Rota-Baxter  operateur of $(A,\ast)$ with respect to $(V,\rho).$
	\end{definition}
	Observe that a linear map $ \mathcal{Z}:V\rightarrow A $ generates a linear deformation of $T$ if only and if : 
	\begin{eqnarray}
		&&\mathcal{Z}u \ast \mathcal{Z}v= \mathcal{Z}(\rho(\mathcal{Z}u)v+\rho(\mathcal{Z}v)u),\label{defor2}\\
		&& Tu\ast \mathcal{Z}v + Tv\ast \mathcal{Z}u-T(\rho(\mathcal{Z}(u))v+\rho(\mathcal{Z}(v))u)-\mathcal{Z}(\rho(T(u))v+\rho(T(v))u)
		=0.\label{defor3}
	\end{eqnarray}
	The identity (\ref{defor2}) means that $\mathcal{Z}$ is a relative Rota-Baxter operateur of $(A,\ast)$ with respect to $(V,\rho).$ Observe that (\ref{defor3}) is equivalent to
	\begin{eqnarray*}
		\mathcal{Z}(u\ast_Tv)=\rho_T(u)\mathcal{Z}v+\rho_T(v)\mathcal{Z}u.
	\end{eqnarray*}
	Therefore,  $(\ref{defor3})$ means that $\mathcal{Z}\in Der(V,A)$ i.e. $\mathcal{Z}$ is a derivation of the Jacobi-Jordan algebra $(V,\ast_T)$ with values in the representation $(A,\rho_T).$
	
	\begin{definition}\label{D9j}
		Let $ T $ and $ T' $ be two relative Rota-Baxter operators on a Jacobi-Jordan algebra $ (A,\cdot) $ with respect to a representation $ (V, \rho) $. A morphism from $ T' $ to $ T $ consits of a Jacobi-Jordan algebras morphism $\phi_{A} : A\to A $ and a linear map $ \phi_{V} : V\to V $ such that:
		\begin{eqnarray}
			\label{36j}
			T \circ \phi_{V}=\phi_{A}\circ T',
		\end{eqnarray}
		\begin{eqnarray}
			\label{37j}
			\phi_{V}\circ\rho(x)=\rho(\phi_{A}(x))\circ\phi_{V} \mbox{ $\forall x\in A $.}
		\end{eqnarray}
		
		In particular, if both $ \phi_A$ and $ \phi_{V}$ are inversible, $ (\phi_{A},\phi_{V}) $ is called an isomorphism from $ T' $ to $ T $.
	\end{definition}
	\begin{proposition}\label{Pro8j}
		Let $ T $ and $ T' $ be two relative Rota-Baxter operators on a Jacobi-Jordan algebra $ (A,\cdot) $ with respect to a representation $ (V; \rho) $ and $ (\phi_{A},\phi_{V}) $ a morphism (resp. an isomorphism) from $ T' $ to $ T $. Then $ \phi_{V} $ is a morphism (resp. an isomorphism) of a Jacobi-Jordan algebra $ (V,\cdot_{T'}) $ to $ (V,\cdot_{T})$.
	\end{proposition}
	\begin{proof}
		Let $ u,v\in V $, then by (\ref{36j}),  we have:\\
		$ \begin{array}{lll}
			\phi_{V} (u\cdot_{T'} v)=\phi_{V}(\rho(T'u)v+\rho(T'v)u)=\phi_{V}(\rho(T'u)v)+\phi_{V}(\rho(T'v)u)\\
			=\rho(\phi_A(T'u))\phi_{V}(v)+\rho(\phi_A(T'v))\phi_{V}(u)=\rho(T(\phi_{V}(u)))\phi_{V}(v)+\rho(T(\phi_{V}(v)))\phi_{V}(u)\\
			=\phi_{V}(u)\cdot_{T}\phi_{V}(v).
		\end{array} $\\
		Then $ \phi_{V} $ is a morphism of left pre-Jacobi-Jordan algebra $ (V,\cdot_{T'}) $ to $ (V,\cdot_{T}) $ .
	\end{proof}
	\begin{definition}\label{D12j}
		Let $ T $ be a relative Rota-Baxter operator on a Jacobi-Jordan algebra $ (A,\ast) $ with respect to a representation $(V;\rho) $. Two linear deformations $ T^{1}_{t}=T+tJ_{1} $ and $ T^{2}_{t}=T+tJ_{2} $ of $T$ are said to be equivalent if there exists $ x\in C^0(V,A) $ such that $ (Id_{A}+tL_{x},Id_{V}+t\rho(x)) $ is a morphism from $ T_{t}^{2} $ to $ T_{t}^{1} $. In particular, a linear deformation $ T_{t}=T+tJ $ of a relative Rota-Baxter operator $ T $ is said to be trivial if there exists $ x\in C^0(V,A) $ such that $ (Id_{A}+tL_{x},Id_{V}+t\rho(x)) $ is a morphism from $ T_{t} $ to $ T $.  	
	\end{definition}
	Suppose that there exists $ x\in C^0(V,A) $ such that $ (Id_{A}+tL_{x},Id_{V}+t\rho(x)) $ is a morphism from $T_{t}^{2} $ to $ T_{t}^{1} $. Then,
	$ Id_{A}+tL_{x} $ is a Jacobi-Jordan algebras morphism of $ (A,\cdot) $ and (\ref{36j})-(\ref{37j}) hold.\\  
	$(Id_{A}+tL_{x})(y\cdot z)=(Id_{A}+tL_{x})(y)\cdot(Id_{A}+tL_{x})(z), \forall y,z\in A $,\\
	if and only if $ x $ satisfies
	\begin{eqnarray}
		&&(x\ast y)\ast (x\ast z)=0,   \, \forall \,\ y, z\in A,\label{39j}\\
		&&(x\ast y)\ast z+(z\ast x)\ast y-(y\ast z)\ast x=0 ,\,\ \forall\,\ y, z\in A.\label{40j}
	\end{eqnarray}
	Then by Eq.(\ref{36j}), we get \\\\
	$(T+tJ_{1})\circ(Id_{V}+t\rho(x))(v)=(Id_{A}+tL_{x})\circ(T+tJ_{2})(v),\hspace*{0.5cm}\forall v\in V$,\\\\
	which holds if and only if  $ x $ satisfies
	\begin{eqnarray}
		&&J_{2}-J_{1}= T\circ\rho(x)-L_{x}\circ T=-\partial_{\rho_{T}}(x), \label{41j}\\
		&&J_{1}\circ\rho(x)=L_{x}\circ J_{2}. \label{42j}
	\end{eqnarray}
	Also by Eq.(\ref{37j}), we obtain \\\\
	$ (Id_{V}+t\rho(x))\circ\rho(y)=\rho(y+tL_{x}(y))\circ(Id_{V}+t\rho(x)), $\\\\
	which holds if and only if $ x $ satisfies
	\begin{eqnarray}
		&&\rho(x\ast y)\circ\rho(x)=0, \,\ \forall y\in A,\label{43j}\\
		&&\rho(x\ast y)+\rho(y)\circ \rho(x)-\rho(x)\circ \rho(y)=0,\,\ \forall y\in A.\label{44j}
	\end{eqnarray}
	Note that Eq. (\ref{41j}) means that $$J_{2}-J_{1}=-\partial_{\rho_{T}}(x).$$ Thus, we have the following result:
	\begin{theorem}
		Let $T$ be a relative Rota-Baxter operator on  a Jacobi-Jordan algebra $(A,\ast)$ with
		respect to a representation $(V,\rho)$. If two linear deformations $T_{t}^1 = T +tJ_1$ and $T_{t}^2= T +tJ_2$ are equivalent, then $J_1$ and $J_2$ are in the same cohomology class of $H^1 (V,A)$.	
	\end{theorem}
	\begin{definition}
		Let $T$ be a relative Rota-Baxter operator on a Jacobi-Jordan algebra $(A,\ast )$ with 
		respect to a representation $(V, \rho).$ An element $x\in C^0(V,A)$ is called a Nijenhuis element associated to $T$ if $x$ satisfies (\ref{39j}), (\ref{40j}), (\ref{43j}), (\ref{44j}) and the equation
		\begin{eqnarray}
			L_x\circ T\circ\rho(x)-L_{x}\circ L_{x}\circ T=0.
			\label{N1}
		\end{eqnarray}
		Denote by $N_{ij}(T )$ the set of Nijenhuis elements associated to a relative Rota-Baxter operator $T.$
	\end{definition}
	The following lemma is very useful to prove the next theorem.
	\begin{lemma}\label{LemT1}
		Let $T$ be a relative Rota-Baxter operator on a Jacobi-Jordan algebra $(A,\cdot)$ with
		respect to a representation $(V,\rho)$. 
		Let $\phi_A: A\rightarrow A$ be a Jacobi-Jordan algebras isomorphism and $\phi_V: V\rightarrow V$
		an isomorphism of vector spaces such that Eqs.
		(\ref{37j}) holds. Then, 
		$	\varphi=\varphi_A^{-1}\circ T\circ\varphi_V$ 
		is a relative Rota-Baxter operator on the on the Jacobi-Jordan algebra $(A,\cdot)$ with
		respect to a representation $(V,\rho)$.
	\end{lemma}
\begin{proof}  First, observe that
\begin{equation}
	\varphi=\varphi_A^{-1}\circ T\circ\varphi_V\iff \varphi_A\circ \varphi=T\circ \varphi_V .
	\label{R1}
\end{equation}		
	Next, let $ u,v\in V$, then we  have\\
	$ \begin{array}{lll}
\varphi(u)\cdot\varphi(v)&=&\bigg(\varphi_A^{-1}\circ T\circ\varphi_V\bigg)(u)\cdot \bigg(\varphi_A^{-1}\circ T\circ\varphi_V\bigg)(v)  \mbox{(by the definition of $ \varphi$)} \\
&=&\varphi_A^{-1}\bigg(T(\varphi_{V}(u))\bigg)\cdot\varphi_A^{-1}\bigg(T(\varphi_{V}(v))\bigg)\\
&=&\varphi_A^{-1}\bigg(T(\varphi_{V}(u))\cdot T(\varphi_{V}(v))\bigg)(\text{ since $ \varphi_A^{-1} $ is a morphism})\\
&=&\varphi_A^{-1}\bigg(T\bigg(\rho(T(\varphi_{V}(u)))\varphi_{V}(v)+\rho(T(\varphi_{V}(v)))\varphi_{V}(u)\bigg)\bigg)\mbox{ (by $ \ref{rbHJJ2}$ )}\\
&=&(\varphi_A^{-1}\circ T)\bigg(\rho(T(\varphi_{V}(u)))\varphi_{V}(v)+\rho(T(\varphi_{V}(v)))\varphi_{V}(u)\bigg)\\
&=&(\varphi_A^{-1}\circ T)\bigg(\rho((T\circ\varphi_{V})(u))\varphi_{V}(v)+\rho((T\circ\varphi_{V})(v))\varphi_{V}(u)\bigg)\\
&=&(\varphi_A^{-1}\circ T)\bigg(\rho((\varphi_A\circ\varphi)(u))\varphi_{V}(v)+\rho((\varphi_A\circ\varphi)(v))\varphi_{V}(u)\bigg) \mbox{(by \ref{R1} )}\\
&=&(\varphi_A^{-1}\circ T)\bigg(\rho((\varphi_A(\varphi(u)))\varphi_{V}(v)+\rho((\varphi_A(\varphi(v)))\varphi_{V}(u)\bigg) \\
	\end{array} $\\
$ \begin{array}{lll}
&=&(\varphi_A^{-1}\circ T)\bigg(\varphi_{V}\circ\rho(\varphi(u))v+\varphi_{V}\circ\rho(\varphi(v))u\bigg)\mbox{ (by \ref{37j})} \\
&=&(\varphi_A^{-1}\circ T\circ \varphi_{V})\bigg(\rho(\varphi(u))v+\rho(\varphi(v))u\bigg)\\
&=&\varphi\bigg(\rho(\varphi(u))v+\rho(\varphi(v))u\bigg).
	\end{array} $ \\
	Therefore, $	\varphi=\varphi_A^{-1}\circ T\circ\varphi_V$ 
	is a relative Rota-Baxter operator on the Jacobi-Jordan algebra $(A,\cdot)$ with
	respect to a representation $(V,\rho)$.	
\end{proof}
	By Eqs. (\ref{39j})-(\ref{44j}), it is obvious that a trivial linear deformation gives rise to a Nijenhuis
element. Conversely, a Nijenhuis element can also generate a trivial linear deformation as the
following theorem shows.
\begin{theorem}\label{thJP}
	Let $T$ be a relative Rota-Baxter operator on a Jacobi-Jordan algebra $(A,\cdot)$ with
	respect to a representation $(V,\rho).$ Then for any $x\in N_{ij}(T ),$ $T_t=T + tJ$ with 
	$J=-\partial_{\rho_{T}}(x)$	is a
	trivial linear deformation of the relative Rota-Baxter operator $T$.
\end{theorem}
\begin{proof}
Suppose that $ T $ be a Rota-Baxter relative operator on a Jacobi-Jordan algebra $ (A,\ast) $ with respect to a representation $ (V,\rho) $ and pick $ (x,v)\in N_{ij}(T)\times V $. Then, we have
$$ 
Jv=-\partial_{\rho_{T}}(x)v
=T(\rho(x)v)-T(v)\ast x
=T\circ\rho(x)v-L_{x}\circ Tv.
 $$
i.e.,
 \begin{eqnarray}
J=T\circ\rho(x)-L_{x}\circ T
\label{J1}.
\end{eqnarray} 
Therefore, by (\ref{N1}), we obtain

 \begin{eqnarray}
L_{x}\circ J=0.
\label{J2}
\end{eqnarray}
Next, from the relation $ T_{t}=T+tJ $ and  (\ref{J2}), we have \begin{eqnarray}
L_{x}\circ T_{t}=L_{x}\circ T.
\label{J3}
\end{eqnarray}
Using again the relation $ T_{t}=T+tJ $, we have\\\\
$ \begin{array}{lll}
T_{t}-T&=&tJ\\
T_{t}-T&=&t\bigg(T\circ\rho(x)-L_{x}\circ T\bigg)\text{( by \ref{J1})}\\
T_{t}-T&=&t\bigg(T\circ\rho(x)-L_{x}\circ T_{t}\bigg)\text{ (by \ref{J3})}\\
T_{t}-T&=&tT\circ\rho(x)-tL_{x}\circ T_{t},\\
T_{t}+tL_{x}\circ T_{t}&=&tT\circ\rho(x)+T.
\end{array} $\\
Thus, \begin{eqnarray}
\bigg(Id_{A}+tL_{x}\bigg)\circ T_{t}=T\circ \bigg(Id_{V}+t\rho(x)\bigg).
\label{J4}
\end{eqnarray}
Let $y,z\in A$ and $ t\in\mathbb{K}.$ By $ (\ref{39j}) $ and $  (\ref{40j}) $, we have \\\\
$ \begin{array}{lll}
y\ast z+tx\ast (y\ast z)=y\ast z+ty\ast(x\ast z)+t(x\ast y)\ast z+t^{2}(x\ast y)\ast(x\ast z)
\end{array} $\\
and therefore, \begin{eqnarray}
\bigg(Id_{A}+tL_{x}\bigg)(y\ast z)=\bigg(Id_{A}+tL_{x}\bigg)(y)\ast \bigg(Id_{A}+tL_{x}\bigg)(z).
\label{J5}
\end{eqnarray}
By (\ref{43j}) and (\ref{44j}), we have\\\\
$ \begin{array}{lll}
\rho(y)+t\rho(x)\circ\rho(y)=\rho(y)+t\rho(y)\circ\rho(x)+t\rho(x\cdot y)+t^{2}t\rho(x\cdot y)\circ\rho(x)
\end{array} $\\
\\
 and therefore, \begin{eqnarray}
 \bigg(Id_{V}+t\rho(x)\bigg)\circ \rho(y)=\rho\bigg(\bigg(Id_{A}+tL_{x}\bigg)(y)\bigg)\circ \bigg(Id_{V}+t\rho(x)\bigg).
 \label{J6}
 \end{eqnarray}
For $ t $ sufficiently small, we see that $ Id_{A}+tL_{x} $ is a Jacobi-Jordan algebras isomorphism and $ Id_{V}+t\rho(x) $
is an isomorphism of vector spaces. Thus, by $ (\ref{J4}) $ we have\\
 $$ T_{t}=\bigg(Id_{A}+tL_{x}\bigg)^{-1}\circ T\circ \bigg(Id_{V}+t\rho(x)\bigg). $$ 
 By Lemma \ref{LemT1}, we deduce that $ T_{t} $ is a relative Rota-Baxter operator on the Jacobi-Jordan algebra
 $ (A,\ast) $ with respect to the representation $ (V,\rho) $, for t sufficiently small. Thus, $ J $ given by Eq. \ref{J1} satisfies the conditions (\ref{defor2}) and (\ref{defor3}).
 Hence, for any $x\in N_{ij}(T ),$ $T_t=T + tJ$ with 
  $J=-\partial_{\rho_{T}}(x)$ 	is a
 trivial linear deformation of the relative Rota-Baxter operator $T$.
\end{proof}

	\begin{definition}
		Let $(A,\ast)$ be a  Jacobi-Jordan algebra. We say that a bilinear map $\psi\in C^2(A,A)$ generate a linear  deformation of the multiplication $\ast$ of the Jacobi-Jordan algebra $(A,\ast)$ if the $t$-parametrized family of bilinear operations
		\begin{eqnarray}
			\mu_t(u, v)=u\ast v+t\psi(u,v) \mbox{ for $t\in\mathbb{K}$ }
		\end{eqnarray}
		gives to $A$ the structure of Jacobi-Jordan algebras i.e. $(A,\mu_t)$ is a Jacobi-Jordan algebra for all $t\in \mathbb{K}.$
	\end{definition}
	By computing the Jordan-Jacobi identity of $\mu_t$ and using Jacobi identity for $\ast$ we obtain:
	\begin{eqnarray}
		&&t\Big(x\ast\psi(y,z)+y\ast\psi(z,x)+z\ast\psi(x,y)+\psi(x,y\ast z)+\psi(y,z\ast x)\nonumber\\
		&&+\psi(z,x\ast y)\Big)+t^2\Big(\psi(\psi(x,y),z)+\psi(\psi(y,z),x)+\psi(\psi(z,x),y)\Big).\nonumber
	\end{eqnarray}
	Hence, $(A,\mu_t)$ is a Jacobi-Jordan algebra for all $t\in \mathbb{K}$ if and only if:
	\begin{eqnarray}
		&&\psi(x,y)=\psi(y,x), \label{defor4b}\\
		&& \psi(\psi(x,y),z)+ \psi(\psi(y,z),x)+ \psi(\psi(z,x),y)=0, \label{defor5}\\
		&&x\ast\psi(y,z)+y\ast\psi(z,x)+z\ast\psi(x,y)\nonumber\\
		&&+\psi(x,y\ast z)+\psi(y,z\ast x)+\psi(z,x\ast y)=0.\label{defor6}
	\end{eqnarray}
	Obviously, (\ref{defor4b}) and (\ref{defor5}) mean that $\psi$ must itself defines a Jacobi algebra structure on $A$.
	Furthermore, (\ref{defor6}) means that $\psi$ is closed with respect to the adjoint representation $(A, L)$ , i.e. $d^2\psi=0.$
	\begin{definition}
		Let $(A, \ast)$ be a Jacobi-Jordan algebra. Two linear deformations 
		$\mu^1_t=\ast+t\psi_1$ and $\mu^2_t=\ast+t\psi_2$ are said to be equivalent if there exists
		a linear operator $N\in gl(A)$ such that 
		$T_t=Id+tN$ is a Jacobi-Jordan algebras
		morphism from $(A, \mu^2_t)$ to $(A, \mu^1_t)$. In particular, a linear deformation
		$\mu_t=\ast+t\psi$ of a Jacobi-Jordan algebra $(A, \ast)$ is said to be trivial if there
		exists a linear operator $N\in gl(A)$ such that for all $t\in\mathbb{K}$
		$T_t=Id+tN$ is  a Jacobi-Jordan algebras
		morphism from $(A, \mu_t)$ to $(A, \ast).$
	\end{definition}
	For all $t\in\mathbb{K}$
	$T_t=Id+tN$ being a
	morphism of algebras is equivalent to
	\begin{eqnarray}
		&&\psi_2(x,y)-\psi_1(x,y)= x\ast N(y)+N(x)\ast y-N(x\ast y), \label{eqdf1}\\
		&&\psi_1(x,N(y))+\psi_1(N(x),y)=N(\psi_2(x,y))-N(x)\ast N(y),\label{eqdf11}\\
		&&\psi_1(N(x),N(y))=0.
	\end{eqnarray}
	Observe that (\ref{eqdf1}) means that $\psi_2-\psi_1=\delta^1 N\in B^2(A,A).$ Hence, it follows:
	\begin{theorem}
		Let $(A, \ast)$ be a Jacobi-Jordan algebra. If two linear deformations 
		$\mu^1_t=\ast+t\psi_1$ and $\mu^2_t=\ast+t\psi_2$ are equivalent, then $\psi_1$ and $\psi_2$ are in the
		same cohomology class of $H^2(A, A)$.
	\end{theorem}
	Observe that a deformation $\mu_t=\ast+t\psi$ of a Jacobi-Jordan algebra $(A, \ast)$ is  trivial i.e., there
	exists a linear operator $N\in gl(A)$ such that for all $t\in\mathbb{K}$
	$T_t=Id+tN$ is  a Jacobi-Jordan algebras
	morphism from $(A, \mu_t)$ to $(A, \ast)$ if and only if
	\begin{eqnarray}
		&&\psi(x,y)= x\ast N(y)+N(x)\ast y-N(x\ast y), \label{eqdft1}\\
		&&N(\psi(x,y))=N(x)\ast N(y).
	\end{eqnarray}
	These two identities allow to give the following definition.
	\begin{definition}
		Let $(A, \ast)$ be a Jacobi-Jordan algebra.
		A linear operator $N\in C^1(A,A)$ is called a Nijienhuis operator if
		\begin{eqnarray}
			N(u)\ast N(v)=N(u\ast_N v), \forall u,v\in A, \label{njj}
		\end{eqnarray}
		where
		\begin{eqnarray}
			u\ast_N v:=N(u)\ast v+u\ast N(v)-N(u\ast v), , \forall u,v\in A.\label{opn}
		\end{eqnarray}
	\end{definition}
	By straightforward computations, we check the following.
	\begin{proposition}\label{njajo}
		Let $(A, \ast)$ be a Jacobi-Jordan algebra and $N\in C^1(A,A)$ be a Nijienhuis operator of $(A, \ast)$. Then,
		$A_N:=(A,\ast_N)$ is a Jacobi-Jordan algebra and $N$ is a morphism of $A_N$ to $(A, \ast)$.
	\end{proposition}
\begin{proof}
	Let $ x,y,z\in A $, denode by $ J_{N} $ the Jacobian of $ A_{N} $ and $ J $ the Jacobian of $ (A,\cdot) $ . Then we have $ x\ast_N y=N(x)\ast y+x\ast N(y)-N(x\ast y) $, $ N(x\ast_N y)=N(x)\ast N(y) $ and \\\\
	$ \begin{array}{lll}
		J_{N}(x,y,z)=\circlearrowleft_{(x,y,z)}(x\ast_Ny)\ast_Nz\\
		=\circlearrowleft_{(x,y,z)}\bigg(N(x\ast_Ny)\ast z+(x\ast_Ny)\ast N(z)-N((x\ast_Ny)\ast z)\bigg)\\
		=\circlearrowleft_{(x,y,z)}\bigg[\bigg(N(x)\ast N(y)\bigg)\ast z+\bigg(N(x)\ast y+x\ast N(y)-N(x\ast y)\bigg)\ast N(z)\\
		-N\bigg(\bigg(N(x)\ast y+x\ast N(y)-N(x\ast y)\bigg)\ast z\bigg)\bigg]\\	
		=	\circlearrowleft_{(x,y,z)}\bigg[\bigg(N(x)\ast N(y)\bigg)\ast z+\bigg(N(x)\ast y\bigg)\ast N(z)+\bigg(x\ast N(y)\bigg)\ast N(z)\\
			\end{array} $\\
	$ \begin{array}{lll}
		-\bigg(N(x\ast y)\ast N(z)\bigg)
		-N\bigg(\bigg(N(x)\ast y+x\ast N(y)-N(x\ast y)\bigg)\ast z\bigg)\bigg]\\
		=	
		\circlearrowleft_{(x,y,z)}\bigg[\bigg(N(x)\ast N(y)\bigg)\ast z+\bigg(N(x)\ast y\bigg)\ast N(z)+\bigg(x\ast N(y)\bigg)\ast N(z)\\
		-N\bigg((x\ast y)\ast_N z\bigg)
		-N\bigg(\bigg(N(x)\ast y+x\ast N(y)-N(x\ast y)\bigg)\ast z\bigg)\bigg]\\
		=
		\circlearrowleft_{(x,y,z)}\bigg[\bigg(N(x)\ast N(y)\bigg)\ast z+\bigg(N(x)\ast y\bigg)\ast N(z)+\bigg(x\ast N(y)\bigg)\ast N(z)\\
		-N\bigg\{(x\ast y)\ast_N z
		+\bigg(N(x)\ast y+x\ast N(y)-N(x\ast y)\bigg)\ast z\bigg\}\bigg]\\
		=	
		\circlearrowleft_{(x,y,z)}\bigg[\bigg(N(x)\ast N(y)\bigg)\ast z+\bigg(N(x)\ast y\bigg)\ast N(z)+\bigg(x\ast N(y)\bigg)\ast N(z)\\
		-N\bigg\{N(x\ast y)\ast z+(x\ast y)\ast N(z)-N\bigg((x\ast y)\ast z\bigg)
		+\bigg(N(x)\ast y+x\ast N(y)-N(x\ast y)\bigg)\ast z\bigg\}\bigg]\\
		=
		\circlearrowleft_{(x,y,z)}\bigg[\bigg(N(x)\ast N(y)\bigg)\ast z+\bigg(N(x)\ast y\bigg)\ast N(z)+\bigg(x\ast N(y)\bigg)\ast N(z)\\
		-N\bigg\{(x\ast y)\ast N(z)
		+\bigg(N(x)\ast y\bigg)\ast z+\bigg (x\ast N(y)\bigg)\ast z-N\bigg((x\ast y)\ast z\bigg)\bigg\}\bigg]\\
		=J\bigg(N(x),N(y),z\bigg)+J\bigg(N(x),y,N(z)\bigg)+J\bigg(x,N(y),N(z)\bigg)\\
		-N\bigg\{J\bigg(x,y,N(z)\bigg)+J\bigg(N(x),y,z\bigg)+J\bigg(x,N(y),z\bigg)-N\bigg(J\bigg(x,y,z\bigg)\bigg)\bigg\}=0.			\end{array} $\\\\
	We have also
	$ x\ast_N y=N(x)\ast y+x\ast N(y)-N(x\ast y )=N(y)\ast x+y\ast N(x)-N(y\ast x )=y\ast_N x$\\
	Therefore $A_N:=(A,\ast_N)$ is a Jacobi-Jordan algebra and $N$ is a morphism of $A_N$ to $(A, \ast)$.
\end{proof}
	\begin{theorem}
		Let $(A, \ast)$ be a Jacobi-Jordan algebra and $N\in C^1(A,A)$ be a Nijienhuis operator. Then, a deformation $\mu_t$ of $(A,\ast)$ can be obtained by putting
		\begin{eqnarray}
			\psi(u,v):=\delta^1 N(u,v)=u\ast_N v.\label{nabil}
		\end{eqnarray}
		Furthermore, this deformation is trivial.
	\end{theorem}
\begin{proof}
	By the definition of $ \psi $, it is clear that $ d^{2}\psi=0 $ since $d^{2}\circ \delta^1=0.$\\
	Therefore $ \psi $ is a $2$-cocycle of the regular representation of Jacobi-Jordan algebra $ (A,\ast). $ Next, $ \psi $ satifies (\ref{defor4b}) and (\ref{defor5}) since $ (A,\ast_N) $ is a Jacobi-Jordan algebra by Proposition \ref{njajo}. Finally, by straightforward computations, we have\\
	\\
	$ \begin{array}{lll}
		(Id+tN)\bigg(x\ast y+t\psi(x,y)\bigg)-(Id+tN)(x)\ast(Id+tN)(y)=x\ast y+tN(x\ast y)+t\psi(x,y)\\
		+t^{2}N(\psi(x,y))-x\ast y-t\bigg(N(x)\ast y+x\ast N(y)\bigg)
		-t^{2}N(x)\ast N(y)\\
		=tN(x\ast y)+t\psi(x,y)
		+t^{2}N(\psi(x,y))-t\bigg(N(x)\ast y+x\ast N(y)\bigg)\\
		-t^{2}N(\psi(x,y)) 
		\text{ (by \ref{nabil} and (\ref{njj})  )}\\
		=tN(x\ast y)+t\psi(x,y)
		-t\bigg(N(x)\ast y+x\ast N(y)\bigg)\\
		=tN(x\ast y)+t\bigg(N(x)\ast y+x\ast N(y)-N(x\ast y)\bigg)
		-t\bigg(N(x)\ast y+x\ast N(y)\bigg)\text{ (by \ref{nabil})}\\
			\end{array} $\\
	$ \begin{array}{lll}
		=t\bigg(N(x)\ast y+x\ast N(y)\bigg)
		-t\bigg(N(x)\ast y+x\ast N(y)\bigg)\\
		=0.
	\end{array} $ \\\\
	It follows that
	$$ (Id+tN)\big(\mu_t(x,y)\big)=(Id+tN)(x)\ast(Id+tN)(y). $$
	Hence , this deformation is trivial.
\end{proof}
	Now, let give the link between these two deformations.
	\begin{proposition}
		Let $(A,\ast)$ be a Jacobi-Jordan algebra and $ T:A \rightarrow A$ be a relative Rota-Baxter operator with respect  to a representation $(V,\rho).$ If a linear map $\mathcal{Z}:V\rightarrow A$ generates a linear deformation of $T$, then the bilinear map 
		$\psi_{\mathcal{Z}} :V^{2}\rightarrow V $ defined by $$\psi_{\mathcal{Z}}(u,v)=\rho(\mathcal{Z}u)v+\rho(\mathcal{Z}v)u, \forall u,v \in V,$$ generates a  linear deformation of the associated Jacobi-Jordan algebra $(V,\ast_{T}).$ 
	\end{proposition}
	\begin{proof}
		Let denote by $\ast_{T_t}$ the corresponding Jacobi-Jordan algebra structure associated to the relative
		Rota-Baxter operator $T_t:=T + t\mathcal{Z}.$ Then we obtain
		\begin{eqnarray*}
			u\ast_{T_t} v&=&\rho(T_tu)v+\rho(T_tv)u\\
			&=&
			\rho(Tu)v+t\rho(\mathcal{Z}u)v+\rho(Tv)u+t\rho(\mathcal{Z}v)u\\
			&=&u\ast_T v+t\psi_{\mathcal{Z}}(u,v) 
			\mbox{\,\ $\forall u,v\in V$.}
		\end{eqnarray*}
		Hence, $\psi_{\mathcal{Z}}$ generates a linear deformation of $(V,\ast_{T}).$
	\end{proof}
	Finally, we obtain the following characterization of relative Rota-Baxter operators that can be esealy checked.
	\begin{proposition}
		Let $\mathcal{A}:=(A, \ast)$ be a Jacobi-Jordan algebra and $\mathcal{V}:=(V,\rho)$ be a representation of $(A, \ast).$
		A linear map $T: V\rightarrow A$ is an
		relative Rota-Baxter operator with respect  to
		$\mathcal{V}$ if and only if $N_T:=\left(
		\begin{array}{cc}
			0& T\\
			0& 0
		\end{array}
		\right)
		: A\oplus V \rightarrow A\oplus V$ is a Nijenhuis operator on
		the Jacobi-Jordan algebra $A\oplus V.$
	\end{proposition}
\begin{proof}
Let $ (x+u,y+v)\in (A\oplus V)^{2}$. Then, 
 we have
$N_{T}(x+u)=T(u)  $, $N_{T}(y+v)=T(v)  $ and
$ \begin{array}{lll}
N_{T}(X)\diamond N_{T}(Y)-N_{T}\bigg(N_{T}(X)\diamond Y+X \diamond N_{T}(Y)-N_{T}(X\diamond Y)\bigg)=0\iff
Tu\ast Tv\\
-N_{T}\bigg(Tu\ast y+\rho(Tu)v+x\ast Tv+\rho(Tv)u-T\bigg(\rho(x)v+\rho(y)u\bigg)\bigg)
=0\\
\iff Tu\ast Tv-T\bigg(\rho(Tu)v+\mu(Tv)u\bigg)=0. 
\end{array} $\\ Hence, the conclusion follows.
\end{proof}
	\section{Cohomologies and linear deformations of a relative Rota-Baxter operators on left pre-Jacobi-Jordan algebras}
	In this section, we consider the notion of left (resp. right) pre-Jacobi-Jordan algebras first introduced in \cite{absb} as left-skew-symmetric (resp. right-skew-symmetric ) algebras. We then study  the relationships with Jacobi-Jordan algebras in terms of relative Rota-Baxter operators of Jacobi-Jordan algebras. 	Next, the cohomology theory of relative Rota-Baxter operators on pre-Jacobi-Jordan algebras is introduced. We use the cohomological approach to study linear and formal deformations of relative Rota-Baxter operators. In particular, the notion of Nijenhuis elements is introduced to characterize trivial linear deformations. 
	
	\subsection{Cohomologies of relative Rota-Baxter operators on pre-Jacobi-Jordan algebras}
	This subsection is devoted to the study of representations and 
	relative Rota-Baxter operators of pre-Jacobi-Jordan algebras.
\begin{definition}\label{D8}
	Let $ (V, \rho,\mu) $ be a representation of a left pre-Jacobi-Jordan algebra $ (A,\cdot) $. A linear map $ T : V\to A $ is called a relative Rota-Baxter operator  of $ A $ associated to $ (V; \rho,\mu) $ if it satisfies
	\begin{eqnarray}
		\label{14}
		T(u)\cdot T(v)=T\biggl(\rho(Tu)v+\mu(Tv)u \biggr),  \text{ for all } u,v\in V.
	\end{eqnarray}
	Note that Rota-Baxter operators on left pre-Jacobi-Jordan algebra of weight $ 0 $ are  relative Rota-Baxter operators  with respect to the regular representation.
\end{definition}	
\begin{example}
	We suppose that $ \mathbb{K}=\mathbb{R} $ . Consider the 2-dimensional left pre-Jacobi-Jordan algebra $ (A,\cdot) $ given with respect to a basis $ \{e_{1},e_{2}\} $ by $$ e_{1}\cdot e_{1}=e_{2}.$$
	Then $T=\begin{pmatrix} x&a\\ y&b \end{pmatrix}  $ is a relative Rota-Baxter operator on $ (A,\cdot) $  with respect to the regular representation if and only if  
	\begin{eqnarray}
		\label{55}
		T(e_{i})\cdot T(e_{j})=T\left(T(e_{i})\cdot e_{j}+e_{i}\cdot T(e_{j}) \right),\forall i,j=1,2.
	\end{eqnarray}
	From $ (\ref{55}) $, we obtain the following system $(S): 
		x^{2}-2xb=
		xa=
		a^{2}=
		xa-ab=0.$ \\
		
Hence, 	$ (S) \iff 
		(a=
		x=0,\ 
		b\in \mathbb{R},\ 
		y\in \mathbb{R})
	\text{ or }
		( a=0,\ 
		x=2b,\ 
		y\in \mathbb{R}).
	$

	Then, for all $ y,b\in\mathbb{R} $, $T_{y,b}=\begin{pmatrix} 0&0\\ y&b \end{pmatrix}  $ and $T'_{y,b}=\begin{pmatrix} 2b&0\\ y&b \end{pmatrix}  $ are relative Rota-Baxter operators on $ (A,\cdot) $  with respect to the regular representation.  
\end{example}
\begin{lemma}
Let $(A, \cdot)$ be a left pre-Jacobi-Jordan algebra and $(V,\rho,\mu)$ be a representation of $(A, \cdot).$ Then, $A\oplus V$ is a representation of $(A, \cdot)$ under the maps $\rho_{A\oplus V},\mu_{A\oplus V}: A\rightarrow gl(A\oplus V)$ defined by
\begin{eqnarray}
	&&\rho_{A\oplus V}(a)(b+v):=a\cdot b+\rho(a)v.\label{sum1}
\end{eqnarray}
\begin{eqnarray}
	&&\mu_{A\oplus V}(a)(b+v):=b\cdot a+\mu(a)v.\label{sum2}
\end{eqnarray}
\end{lemma}
\begin{proof}
Let $ x,y,z\in A $ and $ v\in V $. We have: \\
$ \begin{array}{rcl}
	\rho_{A\oplus V}(x\ast y)(z+v)&&\overset{(\ref{sum1})}{=}(x\ast y)\cdot z+\rho(x\ast y)v\\
	&&=-x\cdot (y\cdot z)-y\cdot (x\cdot z)-\rho(x)\circ\rho (y)v-\rho(y)\circ\rho (x)v\\
	&&=-x\cdot (y\cdot z)-\rho(x)\circ\rho (y)v-y\cdot (x\cdot z)-\rho(y)\circ\rho (x)v\\
	\rho_{A\oplus V}(x\ast y)(z+v)&&=-\rho_{A\oplus V}(x)\circ \rho_{A\oplus V}(y)(z+v)-\rho_{A\oplus V}(y)\circ \rho_{A\oplus V}(x)(z+v)
\end{array} $.\\\\\\
and also,
 \begin{eqnarray}
 	&&	\mu_{A\oplus V}(x\cdot y)(z+v)\overset{(\ref{sum2})}{=}z\cdot(x\cdot y)+\mu(x\cdot y)v
=-(z\cdot x)\cdot y-(x\cdot z)\cdot y-x\cdot (z\cdot y)\nonumber\\
&&-\mu(y)\circ \mu(x)v-\mu(y)\circ \rho(x)v
	-\rho(x)\circ \mu(y)v
	\text{ (by (\ref{HpJJi}) and (\ref{8}) )}\nonumber\\
&&	=\underbrace{-(z\cdot x)\cdot y-\mu(y)\circ \rho(x)v}\underbrace{-x\cdot (z\cdot y)-\rho(x)\circ \mu(y)v}\underbrace{-(x\cdot z)\cdot y-\mu(y)\circ \rho(x)v}\nonumber\\
&&	=-\mu_{A\oplus V}(y)\circ \mu_{A\oplus V}(x)(z+v)-\rho_{A\oplus V}(x)\circ \mu_{A\oplus V}(y)(z+v)-\mu_{A\oplus V}(y)\circ \rho_{A\oplus V}(x)(z+v).\nonumber
\end{eqnarray}
Hence, the conclusions follows. 
\end{proof}
Using the previous lemma, we give
\begin{example}
	Let $(A, \cdot)$ be a left pre-Jacobi-Jordan algebra and $(V,\rho,\mu)$ be a representation of $(A, \cdot).$
	Consider the representation $A\oplus V$  of $(A, \cdot)$ given in the previous lemma and
	define the linear map $T: A\oplus V\rightarrow A, a+v\mapsto a.$ Then $T$ is up to a scalar coefficient, a relative Rota-Baxter operator of $A$ associated to the representation $(A\oplus V,\rho_{A\oplus V}, \mu_{A\oplus V}).$
\end{example}

A relative Rota-Baxter operator of a pre-Jacobi-Jordan algebra turns into a relative Rota-Baxter operator of its associated Jacobi-Jordan algebra as we can see in the following result.
\begin{proposition}
	If $T$ is a relative Rota-Baxter operator of a left pre-Jacobi-Jordan $(A,\cdot)$ with respect to a representation $(V, \rho,\mu)$
	then $T$ is a relative Rota-Baxter operator of its associated Jacobi-Jordan $(A,\ast)$ with respect  to the representation $(V, \rho+\mu).$
\end{proposition}
\begin{proof}
	Suppose that $T$ is a relative Rota-Baxter operator of a left pre-Jacobi-Jordan $(A,\cdot)$ with respect to a representation $(V, \rho,\mu)$. \\
	Let $ u $ and $ v $ be two elements of $ V $. Since  $T$ is a relative Rota-Baxter  operator of a left pre-Jacobi-Jordan $(A,\cdot)$ with respect to a representation $(V, \rho,\mu)$ we have $ Tu, Tv\in A $ and  \\
	$ \begin{array}{lcr}
	Tu\ast Tv= Tu\cdot Tv+Tv\cdot Tu
	=T\biggl(\rho(Tu)v+\mu(Tv)u \biggr)+T\biggl(\rho(Tv)u+\mu(Tu)v \biggr) \\
	=T\biggl((\rho+\mu)(Tu)v+(\rho+\mu)(Tv)u \biggr).
	\end{array}\\
	 $
	 Hence, the conclusion follows.
\end{proof}

\begin{proposition}
	A linear map $T : V\rightarrow A$ is a relative Rota-Baxter operator of a left pre-Jacobi-Jordan  algebra $(A, \cdot)$ with respect to a representation $(V,\rho,\mu)$  if and only if the graph of $T,$
	$$G_r(T):=\{(T(v), v), v\in  V\}$$
	is a subalgebra of the semi-direct product algebra $A\ltimes V.$
\end{proposition}
\begin{proof}
	Let $u,v\in V.$ Then  We compute:
	\begin{eqnarray}
		(Tu,u)\circledast (Tv,v)=(Tu\cdot Tv, \rho(Tu)v+\mu(Tv)u).\nonumber
	\end{eqnarray}
	Hence 
	\begin{eqnarray}
		(Tu,u)\circledast (Tv,v)\in Gr(T)\Leftrightarrow Tu\cdot Tv=T(\rho(Tu)v+\mu(Tv)u).\nonumber
	\end{eqnarray}
\end{proof}
The following result shows that a relative Rota-Baxter operator can be lifted up the
Rota-Baxter operator.
\begin{proposition}
	Let $(A, \cdot)$ be a left pre-Jacobi-Jordan algebra, $(V,\rho,\mu)$ be a representation of $A$ and
	$T : V\rightarrow A$ be a linear map. Define 
	$\widehat{T}\in End(A\oplus V)$ by 
	$\widehat{T}(a+v):=T(v).$ Then T is a relative Rota-Baxter operator with respect to $(V,\rho,\mu)$
	if and only if $\widehat{T}$ is a Rota-Baxter operator on $A\oplus V.$
\end{proposition}
\begin{proof}
	Let $a,b\in A$ and $u,v\in V.$ Then, we have
	\begin{eqnarray}
		\widehat{T}(a+u)\circledast\widehat{T}(b+v)=Tu\circledast Tv=Tu\cdot Tv\nonumber 
	\end{eqnarray}
	and
	\begin{eqnarray}
		&&\widehat{T}(\widehat{T}(a+u)\circledast(b+v)+(a+u)\circledast\widehat{T}(b+v))=\widehat{T}(Tu\cdot b+\rho(Tu)v+a\cdot Tv+\mu(Tv)u)\nonumber\\
		&&=T(\rho(Tu)v+\mu(Tv)u)\nonumber.
	\end{eqnarray}
	Hence,
	\begin{eqnarray}
		&&\widehat{T}(a+u)\circledast\widehat{T}(b+v)=\widehat{T}(\widehat{T}(a+u)\circledast(b+v)+(a+u)\circledast\widehat{T}(b+v))\nonumber
	\end{eqnarray}
	if and only if,
	\begin{eqnarray}
		&&Tu\cdot Tv=T(\rho(Tu)v+\mu(Tv)u).
		\nonumber
	\end{eqnarray}	
\end{proof}
\begin{proposition}\label{HJJpJJ}
	Let $(A,\ast)$ be a Jacobi-Jordan algebra and $(V, \rho)$ be a representation. If $T$ is
	a relative Rota-Baxter operator with respect to $\rho$, then $(V, \cdot)$ is a left pre-Jacobi-Jordan algebra, where
	\begin{eqnarray}
		u\cdot v:=\rho(T(u))v \mbox{ for $u,v\in V$.} \label{revasHJJ1}
	\end{eqnarray}
	Therefore there exists an associated Jacobi-Jordan algebra structure on $V$ given by Eq. (\ref{asHJJ}) and $T$
	is a homomorphism of Jacobi-Jordan algebras. Moreover, $T(V):=\{T(v)|v\in V\} \subset A$ is a Jacobi-Jordan
	subalgebra of $(A, \ast)$ and there is an induced left pre-Jacobi-Jordan algebra structure on $T(V)$ given
	by
	\begin{eqnarray}
		T(u)\bullet T(v):=T(u\ast v) \mbox{ for $u,v\in V$.} \label{revasHJJ2}
	\end{eqnarray}
	The corresponding associated Jacobi-Jordan algebra structure on $T(V)$ given by Eq. (\ref{asHJJ}) is just a
	Jacobi-Jordan subalgebra of $(A, \ast)$ and $T$ is a homomorphism of left pre-Jacobi-Jordan algebras.
\end{proposition}
\begin{proof}
	Let $u, v, w\in V$ and put  
	$u\star v=u\cdot v + v\cdot u.$ Note
	first that $T(u\star v) = T(u)\ast T(v).$ Then, we compute 
	(\ref{HpJJi2}) as follows
	\begin{eqnarray}
		&&(u\star v)\cdot w=\rho(T(u)\ast T(v))w-\rho(T u)\rho(Tv)w-\rho(T v)\rho(tu)w\nonumber\\
		&&=-u\cdot(v\cdot w)-v\cdot(u\cdot w).\nonumber
	\end{eqnarray}
	Therefore, $(V, \cdot)$ is a left pre-Jacobi-Jordan algebra. The other conclusions follow immediately.
\end{proof}
An obvious consequence of Proposition \ref{HJJpJJ} \,\ is the following construction of a left pre-Jacobi-Jordan algebra in terms of a Rota-Baxter operator (of weight zero) of a Jacobi-Jordan algebra.
\begin{corollary}
	Let $(A,\ast)$ be a Jacobi-Jordan algebra and $P$ be a Rota-Baxter operator (of weight
	zero) on $A.$ Then there is a left pre-Jacobi-Jordan algebra structure on $A$ given by
	\begin{eqnarray}
		x\cdot y:=P(x)\ast y \mbox{ for all $x,y\in A.$}
	\end{eqnarray}
\end{corollary}
\begin{proof}
	Straightforward.
\end{proof}
\begin{corollary}
	Let $(A,\ast)$ be a Jacobi-Jordan algebra. Then there exists a compatible left pre-Jacobi-
	Jordan algebra structure on $A$ if and only if there exists an invertible relative Rota-Baxter operator of $(A, \ast).$
\end{corollary}
\begin{proof}
	Let $(A,\cdot)$ be a left pre-Jacobi-Jordan algebra and $(A,\star)$ be the associated Jacobi-Jordan
	algebra. Then the identity map $id : A\rightarrow A$ is an invertible relative Rota-Baxter operator of $(A,\star)$ with respect to $(A, ad).$
	
	Conversely, suppose that there exists an invertible relative Rota-Baxter operator $T$ of $(A, \ast)$ with respect 
	to a representation $(V, \rho),$ then by Proposition \ref{HJJpJJ}, there is a left pre-Jacobi-Jordan algebra
	structure on $T(V)=A$ given by
	$$T(u)\cdot T(v)=T(\rho(T(u))v),\mbox{ for all $u, v \in V.$}$$
	If we set $T(u)=x$ and $T(v)=y,$ then we obtain
	$$ x\cdot y=T(\rho(x)T^{-1}(y)), \mbox{ for all $x, y \in A.$}.$$
	It is a compatible left pre-Jacobi-Jordan algebra structure on $(A,\ast).$ Indeed,
	\begin{eqnarray}
		&&x\cdot y+y\cdot x=T(\rho(x)T^{-1}(y)+\rho(y)T^{-1}(x))\nonumber\\
		&&= T(T^{-1}(x))\ast T(T^{-1}(y))=x\ast y.\nonumber
	\end{eqnarray}
\end{proof}
\begin{proposition}
	Let $T : V\rightarrow A$ be a relative Rota-Baxter operator on the Jacobi-Jordan algebra $( A, \ast)$ with respect to the representation $(V,\rho).$
	Let us define a map $\rho_T : V\rightarrow  gl(A)$ given by
	\begin{eqnarray}
		\rho_T(u)x:=T(u)\ast x-T(\rho(x)u) \mbox{ for all $(u,x)\in V\times A $.}\nonumber
	\end{eqnarray}
	Then, the triplet $(A,\rho_T)$ is a representation of the sub-adjacent Jacobi-Jordan algebra $V^C=(V,\star)$ associated with the left pre-Jacobi-Jordan algebra 
	$(V,\cdot)$ defined in Proposition \ref{HJJpJJ}.
\end{proposition}
\begin{proof}
	Let $u,v\in V,\ x\in A.$ Recall that $u\star v=\rho(Tu)v+\rho(Tv)u$ and $T(u\star v)=T(u)\ast T(v).$ Then, by  straightforward  computaions, we have
	\begin{eqnarray}
		&&\rho_T(u\star v)x=(T(u)\ast T(v))\ast x-T\Big(\rho(x)\rho(Tu)v\Big)-T\Big(\rho(x)\rho(Tv)u\Big).\nonumber
	\end{eqnarray}
	Also, we compute
	\begin{eqnarray}
		&&\rho_T(u)\rho_T(v)x=Tu\ast(Tv\ast x)-Tu\ast T(\rho(x)v)-
		T\Big(\rho(Tv\ast x )u\Big)\nonumber\\
		&&+T\Big( \rho(T(\rho(x)v))u\Big)=
		Tu\ast(Tv\ast x)-T\Big(\rho(Tu)\rho(x)v+\rho(T(\rho(x)v))u \Big)\nonumber\\
		&&-
		T\Big(\rho(T v)\rho(x)u+\rho(x)\rho(Tv)u \Big)+T\Big( \rho(T(\rho(x)v))u\Big)
		\mbox{ (by (\ref{rbHJJ2}) and (\ref{rHJJ2})  ) }\nonumber\\
		&&=
		T u\ast(Tv\ast x)-T\Big(\rho(T u)\rho(x)v\Big)+
		T\Big(\rho(T v)\rho(x)u\Big)+T\Big(\rho(x)\rho(Tv)u \Big).\nonumber
	\end{eqnarray}
	Switching $u$ and $v$ in the above equation, we come to
	\begin{eqnarray}
		\rho_T(v)\rho_T(u)x=
		Tv\ast(Tu\ast x)-T\Big(\rho(T v)\rho(x)u\Big)+
		T\Big(\rho(T u)\rho(x)v\Big)+T\Big(\rho(x)\rho(Tu)v \Big).\nonumber
	\end{eqnarray}
	It follows by (\ref{JJi}) that
	\begin{eqnarray}
		-\rho_T(u)\rho_T(v)x-\rho_T(v)\rho_T(u)x=\rho_T(u\star v)x,\nonumber
	\end{eqnarray}
	i.e., (\ref{rHJJ2}) holds in $(A,\rho_T).$
\end{proof}	
\begin{proposition}	
	\label{l1}
	Let $ T $ be a relative Rota-Baxter operator on a left pre-Jacobi-Jordan algebra $ (A,\cdot) $ with respect to $ (V; \rho,\mu) $. Define
	\begin{eqnarray}
		\label{15}
		u \cdot _{T} v=\rho(Tu)v+\mu(Tv)u,\,\  \forall u,v\in V.
	\end{eqnarray}
	Then $ (V,\cdot _{T} ) $ is a left pre-Jacobi-Jordan algebra.
\end{proposition}	
\begin{proof} Denote by $Aasso_T$ the anti-associator with respect to $\cdot_T$  and
	let $u$, $v$ and $ w $  be three elements of $ V $. \text{ We have } :\\
	\\
	$ \begin{array}{lll}
		Aasso_{T}(u,v,w)+ Aasso_{T}(v,u,w)=(u \cdot _{T} v)\cdot _{T}w+u \cdot _{T} (v\cdot _{T}w)+(v \cdot _{T} u)\cdot _{T}w+v \cdot _{T} (u\cdot _{T}w)\\
		\overset{(\ref{15})}{=}
		\rho(T(u\cdot_T v))w+\mu(Tw)(u\cdot_T v)+\rho(Tu)(v\cdot_T w)+\mu(T(v\cdot_T w))u+
		
		\rho(T(v\cdot_T u))w\\
		+\mu(Tw)(v\cdot_T u)+\rho(Tv)(u\cdot_T w)+\mu(T(u\cdot_T w))v
		\\
		\overset{(\ref{14}),\ref{15}}{=}\rho(Tu\cdot Tv)w+\mu(Tw)\circ\rho(Tu)v+\mu(Tw)\circ\mu(Tv)u+\rho(Tu)\circ\rho(Tv)w\\
		+\rho(Tu)\circ\mu(Tw)v+\mu(Tv\cdot Tw)u+\rho(Tv\cdot Tu)w+\mu(Tw)\circ\rho(Tv)u+\mu(Tw)\circ\mu(Tu)v\\
		+\rho(Tv)\circ\rho(Tu)w+\rho(Tv)\circ\mu(Tw)u+\mu(Tu\cdot Tw)v\\
		=\Big(\rho(Tu\ast Tv)w+\rho(Tu)\circ\rho(Tv)w+\rho(Tv)\circ\rho(Tu)w\Big)+\Big(\mu(Tv\cdot Tw)u\\
		+\mu(Tw)\circ\mu(Tv)u+\mu(Tw)\circ\rho(Tv)u+\rho(Tv)\circ\mu(Tw)u\Big)+\Big(\mu(Tu\cdot Tw)v+\mu(Tw)\circ\mu(Tu)v\\
		+\mu(Tw)\circ\rho(Tu)v+\rho(Tu)\circ\mu(Tw)v\Big)
		\\
		\overset{(\ref{7}),(\ref{8})}{=}0.
	\end{array} $\\
	Therefore, $ (V,\cdot _{T} ) $ is a left pre-Jacobi-Jordan algebra.
\end{proof}
\begin{corollary}
	Let $ T $ be a relative Rota-Baxter operator on a left pre-Jacobi-Jordan algebra $ (A,\cdot) $ with respect to a representation $ (V; \rho,\mu) $. Then $ T $ is a morphism from the left pre-Jacobi-Jordan algebra $ (V,\cdot_{T}) $ to the initial left pre-Jacobi-Jordan algebra $ (A,\cdot) $.  	
\end{corollary}	
\begin{proof}
	It follows from Proposition \ref{l1} and (\ref{14}).
\end{proof}
\begin{theorem}\label{Th1}
	Let $ T $ be a relative Rota-Baxter operator on a left pre-Jacobi-Jordan algebra $ (A,\cdot) $ with respect to a representation $ (V; \rho,\mu) $. Define 
	\begin{eqnarray}
		\label{29} 
		\rho_{T}(v)x=Tv\cdot x-T(\mu(x)v) \mbox{ $\forall (x,v)\in A\otimes V$},\\
		\label{30}
		\mu_{T}(v)x=x\cdot Tv-T(\rho(x)v) \mbox{ $\forall (x,v)\in A\otimes V.$}
	\end{eqnarray}
	Then $ (A; \rho_{T},\mu_{T}) $ is a representation of the left pre-Jacobi-Jordan algebra $ (V,\cdot_{T}) $. 
\end{theorem}
\begin{proof}
	Ideed, for all $ u,v\in V $ and $ x\in A $, we have :\\
	$ \begin{array}{lll}
		\rho_{T}(u\cdot_{T} v+v\cdot_{T} u)x+\rho_{T}(u)\circ \rho_{T}(v)x+\rho_{T}(v)\circ \rho_{T}(u)x\stackrel{(\ref{29}),(\ref{15})}{=}(Tu\cdot Tv)\cdot x\\
		-T\biggl( \mu(x)\circ \rho(Tu)v+\mu(x)\circ \mu(Tv)u  \biggr) +(Tv\cdot Tu)\cdot x-T\biggl(\mu(x)\circ \rho(Tv)u+\mu(x)\circ \mu(Tu)v\biggr)\\
		Tu\cdot (Tv\cdot x)-T\biggl( \rho(Tu)\circ \mu(x)v+\mu(T(\mu(x)v))u+\mu(Tv\cdot x)u-\mu(T(\mu(x)v))u\biggr)\\
		Tv\cdot (Tu\cdot x)-T\biggl( \rho(Tv)\circ \mu(x)u+\mu(T(\mu(x)u))v+\mu(Tu\cdot x)v-\mu(T(\mu(x)u))v\biggr)\\
		=(Tu\cdot Tv)\cdot x+Tu\cdot (Tv\cdot x)+(Tv\cdot Tu)\cdot x+Tv\cdot (Tu\cdot x)-T\biggl( \mu(x)\circ \rho(Tu)v\\+\mu(x)\circ \mu(Tv)u  \biggr)
		-T\biggl( \mu(x)\circ \rho(Tv)u+\mu(x)\circ \mu(Tu)v  \biggr)-T\biggl( \rho(Tu)\circ \mu(x)v+\mu(Tv\cdot x)u\biggr)\\
		-T\biggl( \rho(Tv)\circ \mu(x)u+\mu(Tu\cdot x)v\biggr)\\
		\stackrel{(\ref{antiHas})}{=}Aasso(Tu,Tv,x)+Aasso(Tv,Tu,x)-T\biggl(\mu(Tv\cdot x)u+\mu(x)\circ \mu(Tv)u+\mu(x)\circ \rho(Tv)u\\
		+\rho(Tv)\circ \mu(x)u \biggr)-T\biggl(\mu(Tu\cdot x)v+\mu(x)\circ \mu(Tu)v+\mu(x)\circ \rho(Tu)v
		+\rho(Tu)\circ \mu(x)v\biggr)\\
		\stackrel{(\ref{HpJJi}),(\ref{8})}{=}0.
	\end{array} $\\
	Thus we deduce that $ \rho_{T}(u\cdot_{T} v+v\cdot_{T} u)x+\rho_{T}(u)\circ \rho_{T}(v)x+\rho_{T}(v)\circ \rho_{T}(u)x=0.
	$\\\\	
	We have also:\\\\
	$ \begin{array}{lll}
		\mu_{T}(u\cdot_{T}v)x+\mu_{T}(v)\circ \mu_{T}(u)x+\mu_{T}(v)\circ \rho_{T}(u)x+\rho_{T}(u)\circ \mu_{T}(v)x=x\cdot (Tu\cdot Tv)\\
		-T\biggl(\rho(x)\circ \rho(Tu)v+\rho(x)\circ\mu(Tv)u\biggr)+(x\cdot Tu)\cdot Tv-T\biggl( \rho(T(\rho(x)u))v+\mu(Tv)\circ\rho(x)u\\
		+\rho(x\cdot Tu)v-\rho(T(\rho(x)u))v\biggr)+(Tu\cdot x)\cdot Tv-T\biggl( \rho(T(\mu(x)u))v+\mu(Tv)\circ\mu(x)u
		\\
		+\rho(Tu\cdot x)v-\rho(T(\mu(x)u))v\biggr)
		
		+Tu\cdot(x\cdot Tv)-T\biggl(\rho(Tu)\circ\rho(x)v+\mu(T(\rho(x)v))u\\
		+\mu(x\cdot Tv)u -\mu(T(\rho(x)v))u\biggr)\\
		=(x\cdot Tu)\cdot Tv+x\cdot(Tu\cdot Tv)+(Tu\cdot x)\cdot Tv+Tu\cdot (x\cdot Tv)-T\biggl(\rho(x)\circ \rho(Tu)v\\
		+\rho(x)\circ\mu(Tv)u\biggr)-T\biggl( \mu(Tv)\circ\rho(x)u
		+\rho(x\cdot Tu)v\biggr)-T\biggl( \mu(Tv)\circ\mu(x)u
		+\rho(Tu\cdot x)v\biggr)\\
		-T\biggl(\rho(Tu)\circ\rho(x)v
		+\mu(x\cdot Tv)u\biggr)\\
		\stackrel{(\ref{antiHas})}{=}Aasso(x,Tu,Tv)
		+Aasso(Tu,x,Tv)-T\biggl(\rho(Tu\cdot x+x\cdot Tu)v+\rho(Tu)\circ \rho(x)v
		\\	+\rho(x)\circ \rho(Tu)v\biggr)
		-T\biggl(\mu(x\cdot Tv)u+\mu(Tv)\circ\mu(x)u+\mu(Tv)\circ\rho(x)u+\rho(x)\circ\mu(Tv)u\biggr)\\
		\stackrel{(\ref{HpJJi}),(\ref{7}),(\ref{8})}{=}0.
	\end{array} $\\
	Therefore, $ (A;\rho_{T},\mu_{T}) $ is a representation of the left pre-Jacobi-Jordan algebra $ (V,\cdot_{T}) $.
\end{proof}
	Let $ T $ be a relative Rota-Baxter operator on a left pre-Jacobi-Jordan algebra $ (A,\cdot) $ with respect to a representation $ (V; \rho,\mu) $. Consider the new pre-Jacobi-Jordan algebra $(V,\cdot_T)$ and its representation $ (A;\rho_T,\mu_T)$. Basing in the zigzag complex developped in \cite{aod} for pre-Jacobi-Jordan algebras, we obtain the following zigzag cochain complex   $ (C^{\bullet}(V,A), \mathbf{A}^{\bullet}(V,A),d^{\bullet},\delta^{\bullet}) $ of the algebra $(V,\cdot_T)$ with coefficients in the representation $ (A;\rho_T,\mu_T)$   where\\
for all $ n\in \mathbb{N},n\geq 1 $,  $$\mathbf{C}^{n}(V,A):=\{f:V^{\otimes n}\rightarrow A, f \text{ is linear}\} $$
and $$\mathbf{A}^{n}(V,A) \mbox{ given by}$$
$f\in \mathbf{A}^{n}(V,A)$ if only if $f\in Hom(\wedge^{n-1} V\otimes V, A)$ and 
\begin{eqnarray}
	&& f(u\ast_T v, u_1,\cdots, u_{n-1}, w\cdot_T u_{n})+f(v\ast_T w, u_1,\cdots, u_{n-1}, u\cdot_T u_{n})\nonumber\\
	&&+f(w\ast_T u, u_1,\cdots, u_{n-1}, v\cdot_T u_{n})=0  \mbox{ $\forall (u,v,w, u_1,\cdots, u_n)\in V^{\otimes (n+3)},$}\label{cA}
\end{eqnarray}
\begin{eqnarray}
	\mathbf{A}^{0}(V,A)=\mathbf{C}^{0}(V,A):=\{x\in A,\rho_T(u\cdot_T v)x+\rho_T(u)\circ\rho_T(v)x=0 
	,\forall u,v\in V\}.
\end{eqnarray}
The two sequences of differential maps are given by
linear maps $$  d_T^{n}: \mathbf{C}^{n}(V,A)\rightarrow \mathbf{C}^{n+1}(V,A), $$
$$\delta_T^n : \mathbf{A}^{n}(V,A)\rightarrow \mathbf{C}^{n+1}(V,A),$$
such that
	$\forall n\in\mathbb{N}\setminus\{0\}$, $d_T^n\circ\delta_{T}^{n-1}=0$ where 
\begin{eqnarray}
	&& d_T^n f(u_1, \cdots, u_{n+1})\nonumber\\
	&&=\sum\limits_{i=1}^{n}\rho_T(u_i)f(u_1,\cdots,\widehat{u_i},\cdots, u_{n+1})
	+\sum\limits_{i=1}^{n}\mu_T(u_{n+1})f(u_1,\cdots,\widehat{u_i},\cdots, u_n,u_i)\nonumber\\
	&&+\sum\limits_{i=1}^{n}f(u_1,\cdots,\widehat{u_i},\cdots, u_n,x_i\cdot_T u_{n+1})
	+\sum\limits_{1\leq i< j\leq n} f(u_i\ast_T u_j, u_1, \cdots, \widehat{u_i}, \cdots, \widehat{u_j}, \cdots, u_{n+1}),\nonumber
\end{eqnarray}
\begin{eqnarray}
	&& \delta_T^n f(u_1, \cdots, u_{n+1})\nonumber\\
	&&=\sum\limits_{i=1}^{n}\rho_T(u_i)f(u_1,\cdots,\widehat{u_i},\cdots, u_{n+1})
	+\sum\limits_{i=1}^{n}\mu_T(u_{n+1})f(u_1,\cdots,\widehat{u_i},\cdots, u_n,u_i)\nonumber\\
	&&-\sum\limits_{i=1}^{n}f(u_1,\cdots,\widehat{u_i},\cdots, u_n,u_i\cdot_T u_{n+1})
	-\sum\limits_{1\leq i< j\leq n} f(u_i\ast_T u_j, u_1, \cdots, \widehat{u_i}, \cdots, \widehat{u_j}, \cdots, u_{n+1}), \nonumber
\end{eqnarray}
with \begin{eqnarray}
	\delta_T^{0}x(u)=d_T^{0}x(u):=\rho_T(u)x+\mu_T(u)x,  \forall (x,u)\in A\otimes V.
	\label{delta0}
\end{eqnarray}
\subsection{Linear deformations of relative Rota-Baxter operators on pre-Jacobi-Jordan algebras }
\begin{definition}\label{D10}
	Let $ T $ be a relative Rota-Baxter operator on a left pre-Jacobi-Jordan algebra $ (A,\cdot) $ with respect to a representation $ (V, \rho,\mu) $ and $ J : V\to A $ be a linear map. If $ T_{t}=T+tJ $ is still a relative Rota-Baxter operator on a left pre-Jacobi-Jordan algebra $ (A,\cdot) $ with respect to $ (V, \rho,\mu) $ for all $ t\in \mathbb{K} $, we say that $ J $ generates a linear deformation  of the relative Rota-Baxter operator $ T $.  	
\end{definition}
Observe that $J$ generates  a lineair deformation $T_{t}=T+tJ$ of the relative Rota-Baxter operator $ T $ if and only if for all $ u,v\in V $,
\begin{eqnarray}
	\label{21}
	Ju\cdot Jv=J(\rho(Ju)v+\mu(Jv)u),
\end{eqnarray}
\begin{eqnarray}
	\label{22}
	Tu\cdot Jv+Ju\cdot Tv=T(\rho(Ju)v+\mu(Jv)u)+J(\rho(Tu)v+\mu(Tv)u).
\end{eqnarray}
Eq.(\ref{21}) means that $ J $ is a relative Rota-Baxter operator on a left pre-Jacobi-Jordan algebra $ (A,\cdot) $ with respect to $ (V, \rho,\mu) $\\
Eq.(\ref{22}) is equivalent to
$$\rho_T(u)J(v)+\mu_T(v)J(u)-J(u\cdot_T v)=0,$$
i.e., $J\in Ker(\delta_T^1).$ \\

The two types of linear deformations are related as follows.
\begin{proposition}\label{Pro9}
	If $ J $ generates a linear deformation of a relative Rota-Baxter operator $ T $ on a left pre-Jacobi-Jordan algebra $ (A,\cdot) $ with respect to a representation $(V; \rho,\mu) $, then the product $ \omega_{J} $  defined by
	\begin{eqnarray}
		\label{26}
		\omega_{J}(u,v)=\rho(Ju)v+\rho(Jv)u, \forall u,v\in V,
	\end{eqnarray}
	generates a linear deformation of the left pre-Jacobi-Jordan algebra $ (V,\cdot_{T}) $.  	
\end{proposition}
\begin{proof}
	Suppose that $J$ generates a linear deformation $T_t$ of the relative Rota-Baxter operator $T$  and denote by $ \cdot_{t} $ the corresponding left pre-Jacobi-Jordan algebra structure associated to $T_{t} $. Then we have
	$$
	u\cdot_{t}v=\rho(T_{t}u)v+\mu(T_{t}v)u=u\cdot_{T}v+t\omega_{J}(u,v),\forall u,v\in V.$$
	It follows that $\omega_{J} $ generates a linear deformation of $ (V,\cdot_{T}) $. 
\end{proof}
\begin{definition}\label{D9}
	Let $ T $ and $ T' $ be two relative Rota-Baxter operators on a left pre-Jacobi-Jordan algebra $ (A,\cdot) $ with respect to a representation $ (V, \rho,\mu) $. A morphism from $ T' $ to $ T $ consits of a left pre-Jacobi-Jordan algebras morphism $\phi_{A} : A\to A $ and a linear map $ \phi_{V} : V\to V $ such that:   
	\begin{eqnarray}
		\label{36}
		T \circ \phi_{V}=\phi_{A}\circ T',
	\end{eqnarray}
	\begin{eqnarray}
		\label{37}
		\phi_{V}\circ\rho(x)=\rho(\phi_{A}(x))\circ\phi_{V} \mbox{ $\forall x\in A$},
	\end{eqnarray}
	\begin{eqnarray}
		\label{38}
		\phi_{V}\circ\mu(x)=\mu(\phi_{A}(x))\circ\phi_{V} \mbox{ $\forall x\in A.$}
	\end{eqnarray}	
	
	In particular, if both $ \phi_A$ and $ \phi_{V}$ are inversible, $ (\phi_{A},\phi_{V}) $ is called an isomorphism from $ T' $ to $ T $.
\end{definition}
\begin{proposition}\label{Pro8}
	Let $ T $ and $ T' $ be two relative Rota-Baxter operators on a left pre-Jacobi-Jordan algebra $ (A,\cdot) $ with respect to a representation $ (V; \rho,\mu) $ and $ (\phi_{A},\phi_{V}) $ a morphism (resp. an isomorphism) from $ T' $ to $ T $. Then $ \phi_{V} $ is a morphism (resp. an isomorphism) of left pre-Jacobi-Jordan algebra $ (V,\cdot_{T'}) $ to $ (V,\cdot_{T})$.
\end{proposition}
\begin{proof}
	Let $ u,v\in V $, then by (\ref{36})-(\ref{38}),  we have:\\
	$ \begin{array}{lll}
		\phi_{V} (u\cdot_{T'} v)=\phi_{V}(\rho(T'u)v+\mu(T'v)u)=\phi_{V}(\rho(T'u)v)+\phi_{V}(\mu(T'v)u)\\
		=\rho(\phi_A(T'u))\phi_{V}(v)+\mu(\phi_A(T'v))\phi_{V}(u)=\rho(T(\phi_{V}(u)))\phi_{V}(v)+\mu(T(\phi_{V}(v)))\phi_{V}(u)\\
		=\phi_{V}(u)\cdot_{T}\phi_{V}(v).
	\end{array} $\\
	Then $ \phi_{V} $ is a morphism of left pre-Jacobi-Jordan algebra $ (V,\cdot_{T'}) $ to $ (V,\cdot_{T}) $ .
\end{proof}
\begin{definition}\label{D12}
	Let $ T $ be a relative Rota-Baxter operator on a left pre-Jacobi-Jordan algebra $ (A,\cdot) $ with respect to a representation $(V;\rho,\mu) $. Two linear deformations $ T^{1}_{t}=T+tJ_{1} $ and $ T^{2}_{t}=T+tJ_{2} $ of $T$ are said to be equivalent if there exists $ x\in C^0(V,A) $ such that $ (Id_{A}+tL_{x}+tR_x,Id_{V}+t\rho(x)+t\mu(x)) $ is a morphism from $ T_{t}^{2} $ to $ T_{t}^{1} $. In particular, a linear deformation $ T_{t}=T+tJ $ of a relative Rota-Baxter operator $ T $ is said to be trivial if there exists $ x\in C^0(V,A) $ such that $ (Id_{A}+tL_{x}+tR_x,Id_{V}+t\rho(x)+t\mu(x)) $ is a morphism from $ T_{t} $ to $ T $.  	
\end{definition}
Suppose that there exists $ x\in C^0(V,A) $ such that $ (Id_{A}+tL_{x}+tR_x,Id_{V}+t\rho(x)+t\mu(x)) $ is a morphism from $T_{t}^{2} $ to $ T_{t}^{1} $. Then,
 $ Id_{A}+tL_{x}+tR_x $ is a left pre-Jacobi-Jordan algebras morphism of $ (A,\cdot) $ and (\ref{36})-(\ref{38}) hold.\\  
$(Id_{A}+tL_{x}+tR_x)(y\cdot z)=(Id_{A}+tL_{x}+tR_x)(y)\cdot(Id_{A}+tL_{x}+tR_x)(z), \forall y,z\in A $,\\
if and only if $ x $ satisfies
\begin{eqnarray}
	&&(x\cdot y)\cdot (x\cdot z)+(x\cdot y)\cdot (z\cdot x)\nonumber\\
	&&(y\cdot x)\cdot (x\cdot z)+(y\cdot x)\cdot (z\cdot x)=0, \, \forall y, z\in A,\label{39}\\
	&&(x\cdot y)\cdot z+(y\cdot x)\cdot z+ y\cdot(x\cdot z)+y\cdot(z\cdot x)\nonumber\\
	&&=x\cdot(y\cdot z)+(y\cdot z)\cdot x,\,\ \forall y, z\in A.\label{40}
\end{eqnarray}
Then by Eq.(\ref{36}), we get \\\\
$(T+tJ_{1})\circ(Id_{V}+t\rho(x)+t\mu(x))(v)=(Id_{A}+tL_{x}+tR_x)\circ(T+tJ_{2})(v),\hspace*{0.5cm}\forall v\in V$,\\\\
which holds if and only if  $ x $ satisfies
\begin{eqnarray}
&&J_{2}-J_{1}=T\circ\rho(x)+T\circ\mu(x)-L_{x}\circ T-R_x\circ T=-\partial_T(x), \label{41}\\
&&J_{1}\circ\rho(x)+J_1\circ\mu(x)=L_{x}\circ J_{2}+R_x\circ J_2. \label{42}
\end{eqnarray}
Also by Eq.(\ref{37}), we obtain \\\\
$ (Id_{V}+t\rho(x)+t\mu(x))\circ\rho(y)=\rho(y+tL_{x}(y)+tR_x(y))\circ(Id_{V}+t\rho(x)+t\mu(x)), $\\\\
which holds if and only if $ x $ satisfies
\begin{eqnarray}
&&\rho(x\cdot y)\circ\rho(x)+\rho(y\cdot x)\circ\rho(x)+\rho(x\cdot y)\circ\mu(x)+\rho(y\cdot x)\circ\mu(x)=0, \,\ \forall y\in A,\label{43}\\
&&\rho(x\cdot y)+\rho(y\cdot x)+\rho(y)\circ \rho(x)+\rho(y)\circ \mu(x)=\rho(x)\circ \rho(y)+\mu(x)\circ \rho(y),\,\ \forall y\in A.\label{44}
\end{eqnarray}
Finally, Eq.(\ref{38}) gives \\\\
$ (Id_{V}+t\rho(x)+t\mu(x))\circ\mu(y)=\mu(y+tL_{x}(y)+tR_x(y))\circ(Id_{V}+t\rho(x)+t\mu(x)), \hspace*{0.6cm}\forall y\in A $,\\\\
which holds if and only if  $ x $ satisfies
\begin{eqnarray}
	&&\mu(x\cdot y)\circ\mu(x)+\mu(y\cdot x)\circ\mu(x)+\mu(x\cdot y)\circ\rho(x)+\mu(y\cdot x)\circ\rho(x)=0, \,\ \forall y\in A,\label{45}\\
	&&\mu(x\cdot y)+\mu(y\cdot x)+\mu(y)\circ \mu(x)+\mu(y)\circ \rho(x)=\mu(x)\circ \mu(y)+\rho(x)\circ \mu(y),\,\ \forall y\in A.\label{46}
\end{eqnarray}
Note that Eq. (\ref{41}) means that $$J_{2}-J_{1}=-\partial_T(x).$$ Thus, we have the following result:
\begin{theorem}
Let $T$ be a relative Rota-Baxter operator on  a letf pre-Jacobi-Jordan algebra $(A,\cdot)$ with
respect to a representation $(V,\rho,\mu)$. If two linear deformations $T_{t}^1 = T +tJ_1$ and $T_{t}^2= T +tJ_2$ are equivalent, then $J_1$ and $J_2$ are in the same cohomology class of $H^1 (V,A)$.	
\end{theorem}
\begin{definition}
	Let $T$ be a relative Rota-Baxter operator on a pre-Jacobi-Jordan algebra $(A,\cdot )$ with 
	respect to a representation $(V, \rho,\mu ).$ An element $x\in C^0(V,A)$ is called a Nijenhuis element associated to $T$ if $x$ satisfies (\ref{39}), (\ref{40}), (\ref{43}), (\ref{44}), (\ref{45}), (\ref{46}) and the equation
	\begin{eqnarray}
	(L_x+R_x)\circ T\circ (\rho(x)+\mu(x))-(L_x+R_x)\circ (L_x+R_x)\circ T=0.
	\label{O1}
	\end{eqnarray}
Denote by $N_{ij}(T )$ the set of Nijenhuis elements associated to a relative Rota-Baxter operator $T.$
\end{definition}
The following lemma is very useful.
\begin{lemma}\label{LemT2}
 Let $T$ be a relative Rota-Baxter operator on a left pre-Jacobi-Jordan algebra $(A,\cdot)$ with
 respect to a representation $(V,\rho,\mu)$. 
 Let $\phi_A: A\rightarrow A$ be a pre-Jacobi-Jordan algebras isomorphism and $\phi_V: V\rightarrow V$
  an isomorphism of vector spaces such that Eqs.
  (\ref{37})-(\ref{38}) hold. Then, 
  $	\varphi=\varphi_A^{-1}\circ T\circ\varphi_V$ 
is a relative Rota-Baxter operator on the on the left pre-Jacobi-Jordan algebra $(A,\cdot)$ with
respect to a representation $(V,\rho,\mu)$.
\end{lemma}
\begin{proof}  First, observe that
	\begin{equation}
		\varphi=\varphi_A^{-1}\circ T\circ\varphi_V\iff \varphi_A\circ \varphi=T\circ \varphi_V. 
		\label{R2}
	\end{equation}
Next, let $ u,v\in V$, then we obtain\\
	$ \begin{array}{lll}
	\varphi(u)\cdot\varphi(v)&=&\bigg(\varphi_A^{-1}\circ T\circ\varphi_V\bigg)(u)\cdot \bigg(\varphi_A^{-1}\circ T\circ\varphi_V\bigg)(v)  \mbox{( by the definition of $ \varphi$)} \\
	&=&\varphi_A^{-1}\bigg(T(\varphi_{V}(u))\bigg)\cdot\varphi_A^{-1}\bigg(T(\varphi_{V}(v))\bigg)\\
	&=&\varphi_A^{-1}\bigg(T(\varphi_{V}(u))\cdot T(\varphi_{V}(v))\bigg)(\text{ since $ \varphi_A^{-1} $ is a morphism})\\
	&=&\varphi_A^{-1}\bigg(T\bigg(\rho(T(\varphi_{V}(u)))\varphi_{V}(v)+\mu(T(\varphi_{V}(v)))\varphi_{V}(u)\bigg)\bigg)\mbox{ ( by $ (\ref{14})$ )}\\
	&=&(\varphi_A^{-1}\circ T)\bigg(\rho(T(\varphi_{V}(u)))\varphi_{V}(v)+\mu(T(\varphi_{V}(v)))\varphi_{V}(u)\bigg)\\
	&=&(\varphi_A^{-1}\circ T)\bigg(\rho((T\circ\varphi_{V})(u))\varphi_{V}(v)+\mu((T\circ\varphi_{V})(v))\varphi_{V}(u)\bigg)\\
	&=&(\varphi_A^{-1}\circ T)\bigg(\rho((\varphi_A\circ\varphi)(u))\varphi_{V}(v)+\mu((\varphi_A\circ\varphi)(v))\varphi_{V}(u)\bigg) \mbox{( by (\ref{R2}) )}\\
\end{array} $\\
$ \begin{array}{lll}
	&=&(\varphi_A^{-1}\circ T)\bigg(\rho((\varphi_A(\varphi(u)))\varphi_{V}(v)+\mu((\varphi_A(\varphi(v)))\varphi_{V}(u)\bigg) \\
	&=&(\varphi_A^{-1}\circ T)\bigg(\varphi_{V}\circ\rho(\varphi(u))v+\varphi_{V}\circ\mu(\varphi(v))u\bigg)\mbox{ ( by (\ref{37})-(\ref{38}) )} \\
	&=&(\varphi_A^{-1}\circ T\circ \varphi_{V})\bigg(\rho(\varphi(u))v+\mu(\varphi(v))u\bigg)\\
	&=&\varphi\bigg(\rho(\varphi(u))v+\mu(\varphi(v))u\bigg).
	\end{array} $ \\
	Therefore, $	\varphi=\varphi_A^{-1}\circ T\circ\varphi_V$ 
	is a relative Rota-Baxter operator  on the pre-Jacobi-Jordan algebra $(A,\cdot)$ with
	respect to a representation $(V;\rho,\mu)$. 
\end{proof}
As Jacobi-Jordan algebras case, By Eqs. (\ref{39})-(\ref{46}), it is obvious that a trivial linear deformation gives rise to a Nijenhuis
element. Conversely, a Nijenhuis element can also generate a trivial linear deformation as the
following theorem shows.
\begin{theorem}\label{thJPP}
	Let $T$ be a relative Rota-Baxter operator on a left pre-Jacobi-Jordan algebra $(A,\cdot)$ with
	respect to a representation $(V,\rho,\mu).$ Then for any $x\in N_{ij}(T ),$ $T_t=T + tJ$ with 
	$J=-\partial_T(x)$	is a
	trivial linear deformation of the relative Rota-Baxter operator $T$.
\end{theorem}	
\begin{proof}
Let $ T $ be a Rota-Baxter relative operator on pre-Jacobi-Jordan algebra $ (A,\cdot) $ with respect to a representation $(V;\rho,\mu)$.
Pick $ x\in N_{ij}(T) $ and $ v\in V $ then, we have :\\
$\begin{array}{lll}
Jv&&=-\partial_{T}(x)v\\
&&=-\rho_{T}(v)x-\mu_{T}(v)x \text{( by (\ref{delta0}) ) }
\end{array} $\\
i.e. by (\ref{29})-(\ref{30}), we obtain  \begin{eqnarray}
J=T\circ\rho(x)+T\circ\mu(x)-L_{x}\circ T-R_{x}\circ T.
\label{JJ1}
\end{eqnarray}
Therefore\\\\
$ \begin{array}{lll}
 L_{x}\circ J+R_{x}\circ J&=&L_{x}\circ T\circ\rho(x)+L_{x}\circ T\circ\mu(x)-L_{x}\circ L_{x}\circ T-L_{x}\circ R_{x}\circ T\\
 &+&R_{x}\circ T\circ \rho(x)+R_{x}\circ T\circ\mu(x)-R_{x}\circ L_{x}\circ T-R_{x}\circ R_{x}\circ T\\
 &=&(L_{x}+R_{x})\circ T\circ(\rho(x)+\mu(x))-(L_{x}+R_{x})\circ(R_{x}+L_{x})\circ T\\
\end{array} $\\
i.e.,  by (\ref{O1}), the following holds
\begin{eqnarray}
L_{x}\circ J+R_{x}\circ J=0.
\label{JJ2}
\end{eqnarray}
From the relation $ T_{t}=T+tJ $ we compute \\\\
$ \begin{array}{lll}
L_{x}\circ T_{t}+R_{x}\circ T_{t}&=&L_{x}\circ T+R_{x}\circ T+t(L_{x}\circ J+R_{x}\circ J),
\end{array} $\\\\
i.e., by  (\ref{JJ2}), we get
\begin{eqnarray}
L_{x}\circ T_{t}+R_{x}\circ T_{t}=L_{x}\circ T+R_{x}\circ T
\label{JJ3}
\end{eqnarray}
By $ T_{t}=T+tJ $, we have also\\\\
$ \begin{array}{lll}
T_{t}-T&=&tJ\\
&=&t\bigg(T\circ\rho(x)+T\circ\mu(x)-L_{x}\circ T-R_{x}\circ T\bigg)\text{ (by (\ref{JJ1})  )}\\
&=&t\bigg(T\circ\rho(x)+T\circ\mu(x)-L_{x}\circ T_{t}-R_{x}\circ T_{t}\bigg)\text{ (by (\ref{JJ3}) )}\\
T_{t}+tL_{x}\circ T_{t}+tR_{x}\circ T_{t}&=&T+tT\circ\rho(x)+tT\circ\mu(x)
\end{array} $\\\\
and therefore, the following holds
\begin{eqnarray}
\bigg(Id_{A}+tL_{x}+tR_{x}\bigg)\circ T_{t}=T\circ \bigg(Id_{V}+t\rho(x)+t\mu(x)\bigg).
\label{JJ4}
\end{eqnarray}
Now, let $ y,z\in A $ and $ t\in\mathbb{K} $. Then, by  (\ref{39}) and (\ref{40}), we have\\
\\
$ \begin{array}{lll}
y\cdot z+tx\cdot (y\cdot z)+t(y\cdot z)\cdot x=y\cdot z+ty\cdot (x\cdot z)+ty\cdot(z\cdot x)+t(x\cdot y)\cdot z\\
+t^{2}(x\cdot y)\cdot(x\cdot z)+t^{2}(x\cdot y)\cdot(z\cdot x)+t(y\cdot x)\cdot z+t^{2}(y\cdot x)\cdot(x\cdot z)+t^{2}(y\cdot x)\cdot(z\cdot x)
\end{array} $\\\\
i.e.,
\begin{eqnarray}
\bigg(Id_{A}+tL_{x}+tR_{x}\bigg)(y\cdot z)=\bigg(Id_{A}+tL_{x}+tR_{x}\bigg)(y)\cdot \bigg(Id_{A}+tL_{x}+tR_{x}\bigg)(z).
\end{eqnarray}
Next, by (\ref{43}) and (\ref{44}), we have\\\\
$ \begin{array}{ll}
\rho(y)+t\rho(x)\circ \rho(y)+t\mu(x)\circ \rho(y)=\rho(y)+t\rho(y)\circ \rho(x)+t\rho(y)\circ\mu(x)+t\rho(x\cdot y)\\
+t^{2}\rho(x\cdot y)\circ \rho(x)+t^{2}\rho(x\cdot y)\circ \mu(x)+t\rho(y\cdot x)+t^{2}\rho(y\cdot x)\circ \rho(x)+t^{2}\rho(y\cdot x)\circ \mu(x)
\end{array}\\\\ $
and then,  
\begin{eqnarray}
\big(Id_{V}+t\rho(x)+t\mu(x)\big)\circ \rho(y)=\rho\big(y+tL_{x}(y)+tR_{x}(y)\big)\circ\big(Id_{V}+t\rho(x)+t\mu(x)\big).
\end{eqnarray}
Also, by (\ref{45}) and (\ref{46}), we compute\\\\
$\begin{array}{lll}
\mu(y)+t\rho(x)\circ \mu(y)+t\mu(x)\circ \mu(y)=\mu(y)+t\mu(y)\circ \rho(x)+t\mu(y)\circ \mu(x)+t\mu(x\cdot y)\\
+t^{2}\mu(x\cdot y)\circ \rho(x)
+t^{2}\mu(x\cdot y)\circ \mu(x)+t\mu(y\cdot x)+t^{2}\mu(y\cdot x)\circ\rho(x)+t^{2}\mu(y\cdot x)\circ\mu(x)
\end{array}\\\\  $ 
i.e.,
\begin{eqnarray}
\big(Id_{V}+t\rho(x)+t\mu(x)\big)\circ \mu (y)=\mu\big(y+tL_{x}y+tR_{x}y\big)\circ\big(Id_{V}+t\rho(x)+t\mu(x)\big).
\end{eqnarray}
For $ t $ sufficiently small, we see that $ Id_{A}+tL_{x}+tR_{x} $ is a pre- Jacobi-Jordan algebras isomorphism and $ Id_{V}+t\rho(x)+t\mu(x) $
is an isomorphism of vector spaces. Thus, by $ (\ref{JJ4}) $ we have\\
$$ T_{t}=\bigg(Id_{A}+tL_{x}+tR_{x}\bigg)^{-1}\circ T\circ \bigg(Id_{V}+t\rho(x)+t\mu(x)\bigg).$$ 
By Lemma \ref{LemT2}, we deduce that $ T_{t} $ is a relative Rota-Baxter operator on the pre-Jacobi-Jordan algebra
$ (A,\cdot) $ with respect to the representation $ (V;\rho,\mu) $, for t sufficiently small. Hence, $ J $ given by Eq. \ref{JJ1} satisfies the conditions (\ref{21}) and (\ref{22}). It follows that, for any $x\in N_{ij}(T ),$ $T_t=T + tJ$ with 
$J=-\partial_{T}(x)$ is a
trivial linear deformation of the relative Rota-Baxter operator $T$.
\end{proof}

\section{Nijenhuis operators and relative Rota-Baxter operators}
\begin{definition}\cite{aod}\label{D13}
	Let $ (A,\cdot) $ be a left pre-Jacobi-Jordan algebra. A linear map $ N : A\to A $ is said to be a Nijenhuis operator if 
	\begin{eqnarray}
		\label{9}
		N(x)\cdot N(y)=N\bigg(N(x)\cdot y+x\cdot N(y)-N(x\cdot y)\bigg),&\forall x,y\in A.
	\end{eqnarray}
\end{definition}
\begin{remark}
	Note that a Rota-Baxter operator of weight -1 on left pre-Jacobi-Jordan algebra $ A $ is exactly an Nijenhuis operator.
\end{remark}
\begin{lemma} \label{l3}	
	Let $ (A,\cdot) $ be a left  pre-Jacobi-Jordan algebra and $ (V;\rho,\mu) $ be a representation of $A$. A linear map $ T : V\to A $ is a relative Rota-Baxter operator on $ (A,\cdot) $ with respect  to $ (V;\rho,\mu) $ if and only if for any $ \lambda \in \mathbb{K} $, the linear map $ N_{T}=\begin{pmatrix} 0&T\\ 0&-\lambda Id_{V} \end{pmatrix} :  A\oplus V\to A\oplus V $ is a Nijenhuis operator on the semi-direct  $ (A\oplus V,\circledast) $ of  $ (A,\cdot) $ by $ (V;\rho,\mu) $ , where $ \circledast $ is given by $ (\ref{spjj})$.	
\end{lemma}
\begin{proof}
	Let $ X=(x,u) $ and $ Y=(y,v) $ be two elements of $ A\oplus V $. We have\\
	$ \begin{array}{lll}
		N_{T}(X)\circledast N_{T}(Y)-N_{T}\bigg(N_{T}(X)\circledast Y+X \circledast N_{T}(Y)-N_{T}(X\circledast Y)\bigg)=0\\
		\iff \bigg(Tu\cdot Tv,-\lambda\rho(Tu)v-\lambda\mu(Tv)u\bigg)-N_{T}\bigg(x\cdot Tv+Tu\cdot y-T(\rho(x)v+\mu(y)u), \rho(Tu)v\\
		+\mu(Tv)u\bigg)
		=0\\
		\iff\bigg(Tu\cdot Tv,-\lambda\rho(Tu)v-\lambda\mu(Tv)u\bigg)-\bigg(T\bigg(\rho(Tu)v+\mu(Tv)u\bigg),-\lambda\rho(Tu)v\\
		-\lambda\mu(Tv)u\bigg)=0.
	\end{array} $\\
	Therefore,  For any $ \lambda \in \mathbb{K} $, $ T $ is a relative Rota-Baxter operator on $ A $ with respect  to $ (V;\rho,\mu) $ if and only if the linear map $ N_{T}$ is a Nijenhuis operator on the semidirect product left pre-Jacobi-Jordan algebra $ (A\oplus V,\circledast)$.  
\end{proof}
\begin{proposition}\label{Pro12} Let $ (V;\rho,\mu) $ be a representation of a left pre-Jacobi-Jordan algebra $ (A,\cdot) $. Let $ T : V\to A $ be a linear map. Then the following statements are equivalent: 
	\begin{enumerate}
		\item [(i)]$ T $ is a relative Rota-Baxter operator on $ (A,\cdot) $ with respect to $(V;\rho,\mu)$;
		\item [(ii)]$  N_{T}=\begin{pmatrix} 0&T\\ 0&0 \end{pmatrix}  $ is a Nijenhuis operator on the left pre-Jacobi-Jordan algebra $ (A\oplus V,\circledast)$;
		\item [(iii)]$  N_{T}=\begin{pmatrix} 0&T\\ 0&Id_{V}\end{pmatrix}  $ is a Nijenhuis operator on the left pre-Jacobi-Jordan algebra\\
		$ (A\oplus V,\circledast)$.
	\end{enumerate} 
\end{proposition}
\begin{proof}
	It follows from Lemma \ref{l3}.
\end{proof}
\section{Compatible relative Rota-Baxter operators on pre-jacobi-jardan algebra}
\begin{definition}Let $ (A,\cdot) $ be a left pre-Jacobi-Jordan algebra and $ (V;\rho,\mu) $ be a representation of $A$. Let $ T_{1} $, $ T_{2}:V\to A $ be two relative Rota-Baxter operators with respect to $ (V;\rho,\mu) $. If for all $ k_{1} $, $ k_{2}\in \mathbb{K}\setminus\{0\},$  $k_{1}T_{1}+k_{2}T_{2}  $ is still a relative Rota-Baxter operator with respect to $ (V;\rho,\mu) $, then $ T_{1} $ and $ T_{2} $ are called compatible and we denote $ T_{1} \sim T_{2}  $.     
\end{definition}
\begin{proposition}
	Let $ T_{1} $, $ T_{2}:V\to A $ be two relative Rota-Baxter operators on a left pre-Jacobi-Jordan algebra $ (A,\cdot) $ with respect to a representation $ (V;\rho,\mu) $. Then $ T_{1} $ and $ T_{2} $ are compatible if and only if the following equation holds 
	\begin{eqnarray}
		&& T_{1}(u)\cdot T_{2}(v)+T_{2}(u)\cdot T_{1}(v)\nonumber\\
		&&=T_{1}\bigg(\rho(T_{2}u)v+\mu(T_{2}u)v\bigg)+T_{2}\bigg(\rho(T_{1}u)v+\mu(T_{1}u)v\bigg), \forall u,v\in V. \label{34} 
	\end{eqnarray} 
\end{proposition}
\begin{proof} Let $ k_{1},k_{2}\in \mathbb{K}\setminus\{0\} $ and set $ T=k_{1}T_{1}+k_{2}T_{2} $.
	Then, for all  $ u,v\in V $, we have:\\
	$ \begin{array}{lll}
		T_{1}\sim T_{2}\iff T(u)\cdot T(v)=T\bigg(\rho(Tu)v+\mu(Tu)v\bigg)\\
		\iff \bigg(k_{1}T_{1}(u)+k_{2}T_{2}(u)\bigg)\cdot \bigg(k_{1}T_{1}(v)+k_{2}T_{2}(v)\bigg)=T\bigg(k_{1}\rho(T_{1}u)v+k_{2}\rho(T_{2}u)v\\
		+k_{1}\mu(T_{1}u)v+k_{2}\mu(T_{2}u)v\bigg)\\
			\end{array} $\\
	$ \begin{array}{lll}
		\iff k^{2}_{1}T_{1}(u)\cdot T_{1}(v)+k_{1}k_{2}T_{1}(u)\cdot T_{2}(v)+k_{1}k_{2}T_{2}(u)\cdot T_{1}(v)+k^{2}_{2}T_{2}(u)\cdot T_{2}(v)=\\
		T\bigg(k_{1}\bigg(\rho(T_{1}u)v+\mu(T_{1}v)u\bigg)+k_{2}\bigg(\rho(T_{2}u)v+\mu(T_{2}v)u\bigg) \bigg)\\
		\iff k^{2}_{1}T_{1}(u)\cdot T_{1}(v)+k_{1}k_{2}T_{1}(u)\cdot T_{2}(v)+k_{1}k_{2}T_{2}(u)\cdot T_{1}(v)+k^{2}_{2}T_{2}(u)\cdot T_{2}(v)=\\
		k_{1}T_{1}\bigg(k_{1}\bigg(\rho(T_{1}u)v+\mu(T_{1}v)u\bigg)+k_{2}\bigg(\rho(T_{2}u)v+\mu(T_{2}v)u\bigg) \bigg)\\
		+k_{2}T_{2}\bigg(k_{1}\bigg(\rho(T_{1}u)v+\mu(T_{1}v)u\bigg)+k_{2}\bigg(\rho(T_{2}u)v+\mu(T_{2}v)u\bigg) \bigg)\\
		\iff k^{2}_{1}T_{1}(u)\cdot T_{1}(v)+k_{1}k_{2}T_{1}(u)\cdot T_{2}(v)+k_{1}k_{2}T_{2}(u)\cdot T_{1}(v)+k^{2}_{2}T_{2}(u)\cdot T_{2}(v)=\\
		k_{1}^{2}T_{1}\bigg(\rho(T_{1}u)v+\mu(T_{1}v)u\bigg)
		
		+k_{1}k_{2}T_{1}\bigg(\rho(T_{2}u)v+\mu(T_{2}v)u\bigg)\\
		+k_{2}k_{1}T_{2}\bigg(\rho(T_{1}u)v+\mu(T_{1}v)u\bigg)
		+k_{2}k_{1}T_{2}\bigg(\rho(T_{2}u)v+\mu(T_{2}v)u\bigg)\\
		\iff k_{1}k_{2}\bigg(T_{1}(u)\cdot T_{2}(v)+T_{2}(u)\cdot T_{1}(v)\bigg)=\\
		k_{2}k_{1}\bigg(T_{2}\bigg(\rho(T_{1}u)v+\mu(T_{1}v)u\bigg)+T_{1}\bigg(\rho(T_{2}u)v+\mu(T_{2}v)u\bigg)\bigg)
		 \nonumber\\
		( \text{ since } T_{1} \text{ and }T_{2} \text{ are relative  Rota-Baxter operators} )\\
		\iff T_{1}(u)\cdot T_{2}(v)+T_{2}(u)\cdot T_{1}(v)=
		T_{2}\bigg(\rho(T_{1}u)v+\mu(T_{1}v)u\bigg)+T_{1}\bigg(\rho(T_{2}u)v+\mu(T_{2}v)u\bigg).
	\end{array} $\\
\end{proof}
\begin{remark}
	Eq.(\ref{22}) means that $ J $ and $ T $ are compatible. 
\end{remark}
\begin{proposition}\label{Pro3}
	Let  $ T :V\to A $ be a relative Rota-Baxter operator on a left pre-Jacobi-Jordan algebra $ (A,\cdot) $ with respect to a representation $ (V;\rho,\mu) $ and $ N: A\to A $ be a Nijenhuis operator  on $ (A,\cdot) $. Then $ NT $ is a relative Rota-Baxter operator on a left pre-Jacobi-Jordan algebra $ (A,\cdot) $ with respect to the representation $ (V;\rho,\mu) $ if and only if the following equation holds
	\begin{eqnarray}
		\label{35}
		\begin{array}{lll}
			&&N\bigg((NT)u\cdot T(v)+T(u)\cdot (NT)(v) \bigg)\\
			&&=N\bigg(T\bigg(\rho(NTu)v+\mu(NTv)u\bigg)\bigg)
			+N\bigg(NT\bigg(\rho(Tu)v+\mu(Tv)u\bigg)\bigg),\,\ 
			\forall  u,v\in V.
		\end{array}
	\end{eqnarray}
\end{proposition}
\begin{proof}
	Let $ u,v\in V $. We have\\
	$ \begin{array}{lll}
		\text{$NT$ is  a relative Rota-Baxter operator }\iff NT(u)\cdot NT(v)=NT\bigg(\rho(NT(u))v+\mu(NT(v))u\bigg)\\
		\iff N\bigg(NT(u)\cdot T(v)+T(u)\cdot NT(v)-N(Tu\cdot Tv)\bigg)\\
		=N\bigg(T\bigg(\rho(NT(u))v+\mu(NT(v))u\bigg)\bigg)
		
		\mbox{ ( since $N$ is a Nijenhuis operator )}\\
		\iff N\bigg(NT(u)\cdot T(v)+T(u)\cdot NT(v)\bigg)-N^{2}(Tu\cdot Tv)=N\bigg(T\bigg(\rho(NT(u))v+\mu(NT(v))u\bigg)\bigg)\\
		\iff N\bigg(NT(u)\cdot T(v)+T(u)\cdot NT(v)\bigg)-N\bigg(NT\bigg(\rho(Tu)v+\mu(Tv)u\bigg)\bigg)=\\
			\end{array} $\\
	$ \begin{array}{lll}
		N\bigg(T\bigg(\rho(NT(u))v+\mu(NT(v))u\bigg)\bigg)
		\mbox{  ( since $T$ is a relative Rota-Baxter operator )}
		\\
		\iff(\ref{35}).
	\end{array} $\\
\end{proof}
\begin{corollary}
	Let  $ T :V\to A $ be a relative Rota-Baxter operator on a left pre-Jacobi-Jordan algebra $ (A,\cdot) $ with respect to a representation $ (V;\rho,\mu) $ and $ N: A\to A $ be a Nijenhuis operator  on $ (A,\cdot) $.
	If $ N $ is inversible, then $ T $ and $ NT $ are compatible.
\end{corollary}
\begin{proof}
	It follows from Proposition \ref{Pro3}.
\end{proof}
A pair of compatible relative Rota-Baxter operators can also give rise to a Nijenhuis operator under some conditions.
\begin{proposition}
	Let  $ T_{1},T_{2} :V\to A $ be relative Rota-Baxter operators on a left pre-Jacobi-Jordan algebra $ (A,\cdot) $ with respect to a representation $ (V;\rho,\mu) $ and $ N: A\to A $ be a linear map.
	Suppose that $ T_{2} $ is invertible and $T_{1}=N\circ T_{2}  $.
	If $ T_{1} $ and $ T_{2} $ are compatible, then $ N $ is a Nijenhuis operator on the left pre-Jacobi-Jordan algebra $ (A,\cdot). $
\end{proposition}
\begin{proof}
	Let $ x,y\in A $, then there exist $ u,v\in V $ such that $ T_{2}(u)=x $ and $ T_{2}(v)=y $.\\
	$ \begin{array}{lll}
		T_{1}  \text{ and } T_{2} \text{ are compatible}\iff N\circ T_{2}  \text{ and } T_{2} \text{ are compatible}\\
		\iff NT_{2}(u)\cdot T_{2}(v)+T_{2}(u)\cdot NT_{2}(v)=NT_{2}\bigg(\rho(T_{2}u)v+\mu(T_{2}v)u\bigg)\\
		+T_{2}\bigg(\rho(NT_{2}u)v+\mu(NT_{2}v)u\bigg)\\
		\iff N(x)\cdot y+x\cdot N(y)=N\bigg(T_{2}(u)\cdot T_{2}(v)\bigg)+T_{2}\bigg(\rho(T_{1}u)v+\mu(T_{1}v)u\bigg)\\
		\iff N(x)\cdot y+x\cdot N(y)=N(x\cdot y)+T_{2}\bigg(\rho(T_{1}u)v+\mu(T_{1}v)u\bigg)\\
		\iff  N(x)\cdot y+x\cdot N(y)-N(x\cdot y)=T_{2}\bigg(\rho(T_{1}u)v+\mu(T_{1}v)u\bigg)\\
		\implies N\bigg( N(x)\cdot y+x\cdot N(y)-N(x\cdot y)\bigg)=NT_{2}\bigg(\rho(T_{1}u)v+\mu(T_{1}v)u\bigg)\\
		\implies N\bigg( N(x)\cdot y+x\cdot N(y)-N(x\cdot y)\bigg)=T_{1}\bigg(\rho(T_{1}u)v+\mu(T_{1}v)u\bigg)\\
		\implies N\bigg( N(x)\cdot y+x\cdot N(y)-N(x\cdot y)\bigg)=T_{1}(u)\cdot T_{1}(v)\\
		\implies N\bigg( N(x)\cdot y+x\cdot N(y)-N(x\cdot y)\bigg)=NT_{2}(u)\cdot NT_{2}(v)\\
		\implies N\bigg( N(x)\cdot y+x\cdot N(y)-N(x\cdot y)\bigg)=N(x)\cdot N(y)\\
		\implies \text{N is a Nijenhuis operator on the left pre-Jacobi-Jordan.}
	\end{array} $ \\
\end{proof}

\end{document}